\documentclass[11pt,a4paper,leqno]{article}
\usepackage{a4wide}
\usepackage{latexsym} 
\usepackage{amsmath}
\usepackage{amssymb}
\usepackage{stackrel} 
\usepackage{color}
\usepackage{graphicx}

\usepackage{hyperref}

\newtheorem{defin}{Definition}
\newtheorem{lemma}{Lemma}
\newtheorem{prop}{Proposition}
\newtheorem{theo}{Theorem}

\newenvironment{proof}{\medskip\par\noindent{\bf Proof}}{\hfill $\Box$
\medskip\par}

\newcommand{\C}{\mathbb{C}}
\newcommand{\N}{\mathbb{N}}
\newcommand{\Q}{\mathbb{Q}}
\newcommand{\R}{\mathbb{R}}
\newcommand{\Z}{\mathbb{Z}}

\newcommand{\bT}{\boldsymbol{T}}

\newcommand{\bt}{\boldsymbol{t}}

\begin{document}
\title{On a $q-$analog of a singularly perturbed problem of irregular type with two complex time variables}
\author{{\bf A. Lastra\footnote{The author is partially supported by the project MTM2016-77642-C2-1-P of Ministerio de Econom\'ia y Competitividad, Spain}, S. Malek\footnote{The author is partially supported by the project MTM2016-77642-C2-1-P of Ministerio de Econom\'ia y Competitividad, Spain.}}\\
University of Alcal\'{a}, Departamento de F\'{i}sica y Matem\'{a}ticas,\\
Ap. de Correos 20, E-28871 Alcal\'{a} de Henares (Madrid), Spain,\\
University of Lille 1, Laboratoire Paul Painlev\'e,\\
59655 Villeneuve d'Ascq cedex, France,\\
{\tt alberto.lastra@uah.es}\\
{\tt Stephane.Malek@math.univ-lille1.fr }}
\date{}
\maketitle
\thispagestyle{empty}
{ \small \begin{center}
{\bf Abstract}
\end{center}

Analytic solutions and their formal asymptotic expansions for a family of the singularly perturbed $q-$difference-differential equations in the complex domain are constructed. They stand for a $q-$analog of the singularly perturbed partial differential equations considered in~\cite{family3}. In the present work, we construct outer and inner analytic solutions of the main equation, each of them showing asymptotic expansions of essentially different nature with respect to the perturbation parameter. The appearance of the $-1$-branch of Lambert $W$ function will be crucial in this respect.

\medskip

}
\bigskip \bigskip

\section{Introduction}

{\color{blue}{



}}

The intent of this study is to provide analytic solutions
and their parametric asymptotic expansions to a family of singularly perturbed $q-$difference-differential equations in the complex domain, in which two time variables act, for some fixed $q>1$. 

More precisely, we consider equations of the form

\begin{equation}\label{epral}
Q(\partial_z)u(\bt,z,\epsilon)=P(\bt,z,\epsilon,\partial_{t_2},\sigma_{q;t_1},\sigma_{q;t_2})u(\bt,z,\epsilon) +f(\bt,z,\epsilon),
\end{equation}
under null initial data $u(t_1,0,z,\epsilon)\equiv u(0,t_2,z,\epsilon)\equiv 0$. Here $Q(X)\in\C[X]$, and $P$ stands for a polynomial with complex coefficients with respect to $\bt:=(t_1,t_2),\partial_{t_2}$, a polynomial with rational powers with respect to $\sigma_{q;t_1}$ and $\sigma_{q;t_2}$; holomorphic with respect to $z$ on a horizontal strip $H_\beta:=\{z\in\C:|\hbox{Im}(z)|<\beta\}$ for some $\beta>0$, and holomorphic with respect to the perturbation parameter $\epsilon$ on a small disc centered at the origin, say $D(0,\epsilon_0)$ for some $\epsilon_0>0$. Throughout the present work, $\sigma_{q;t}$ stands for the dilation operator on $t$ variable for some fixed $q>1$, i.e.
$$\sigma_{q;t}(f(t)):=f(qt).$$
We adopt the following notation $\sigma_{q;t}^\delta(f(t))=f(q^\delta t)$, for any $\delta\in\Q$.

The forcing term $f(\bt,z,\epsilon)$ is constructed under certain growth conditions, and turns out to be a holomorphic function in $\C^{\star}\times \C^{\star}\times H_{\beta'}\times (D(0,\epsilon_0)\setminus\{0\})$, for $0<\beta'<\beta$.

The precise assumptions on the elements involved in the main equation under study are detailed in Section~\ref{sec2}.

The problem under study (\ref{epral}) turns out to be a $q-$analog of the main problem studied in~\cite{family3}, 
\begin{equation}\label{epral2}
Q(\partial_z)u(\bt,z,\epsilon)=P(t_1^{k_1+1}\partial_{t_1},t_2^{k_2+1}\partial_{t_2},\partial_z,z,\epsilon)u(\bt,z,\epsilon)+f(\bt,z,\epsilon),
\end{equation}
under null initial data $u(t_1,0,z,\epsilon)\equiv u(0,t_2,z,\epsilon)\equiv 0$, and where $Q(X)\in\C[X]$, and $P$ is a polynomial with respect to its first three variables, with holomorphic coefficients on $H_\beta\times D(0,\epsilon_0)$, and the forcing term is holomorphic on $\C^\star\times \C\times H_{\beta'}\times (D(0,\epsilon_0)\setminus\{0\})$. The analytic solutions and their asymptotic expansions of those singularly perturbed partial differential equations are obtained in~\cite{family3}. 
More precisely, the so-called inner solutions are holomorphic solutions of (\ref{epral2}), holomorphic on domains in time which depend on the perturbation parameter and approach infinity, admits Gevrey
asymptotic expansion of certain positive order, with respect to $\epsilon$, whereas the so-called outer solutions are holomorphic solutions of (\ref{epral2}), holomorphic on a product of finite sectors with vertex at the origin with respect to the time variables, admit Gevrey asymptotic expansion of a different positive order, with respect to $\epsilon$. 

The previous phenomena of existence of different asymptotic expansions regarding different domains of the actual solution of the main problem is enhanced in the present work, in the sense that different nature on the asymptotic expansions are observed: Gevrey and $q-$Gevrey asymptotic expansions.

In this work, we search for the analytic solutions of the main problem as the inverse Fourier transform and $q-$Laplace transform of a positive order in the form
\begin{equation}\label{e111}
u(\bt,z,\epsilon)=\frac{1}{(2\pi)^{1/2}\pi_{q^{1/k_1}}}\int_{-\infty}^{\infty}\int_{L_\gamma}\omega(u,m,\epsilon)\frac{1}{\Theta_{q^{1/k_1}}\left(\frac{u}{\epsilon^{\lambda_1}t_1}\right)}e^{-\left(\frac{1}{\epsilon^{\lambda_2}t_2}\right)^{k_2}u}e^{izm}dm\frac{du}{u},
\end{equation}
 (see (\ref{e225}) and (\ref{e242})), for some appropriate $\gamma\in\R$ and $k_1,k_2,\lambda_1,\lambda_2>0$ (see Section~\ref{sec2} for their definition). The function $\omega(\tau,m,\epsilon)$ is obtained from a fixed point argument (see Proposition~\ref{prop5}), and belongs to $q\hbox{Exp}^d_{(k'_1,\beta,\mu,\alpha)}$, a Banach space of holomorphic functions with $q-$exponential growth and exponential decay with respect to $\tau$ and $m$, respectively (see Definition~\ref{defi11}).

The form (\ref{e111}) of the analytic solutions is motivated on the shape of those of the main problem in~\cite{family3}, mixing both time variables in a common Laplace operator. 


A first family of analytic solutions of (\ref{epral}) is constructed on domains of the form $\mathcal{T}_{1}\times\mathcal{T}_{2,\epsilon}\times H_{\beta'}\times \mathcal{E}_{h_1}^{\infty}$, and a second on domains of the form $\mathcal{T}_1\times(\mathcal{T}_2\cap D(0,\rho_2))\times H_{\beta'}\times \mathcal{E}^0_{h_2}$, where $\mathcal{T}_1$ is a finite sector, $\mathcal{T}_{2}$ is an unbounded sector and where $\mathcal{T}_{2,\epsilon}\subseteq\mathcal{T}_2$ is a bounded sector which depends on $\epsilon\in\mathcal{E}_{h_1}^{\infty}$, and tends to infinity with $\epsilon$ approaching the origin. The sets $\underline{\mathcal{E}}^0=(\mathcal{E}_{h_2}^0)_{0\le h_2\le \iota_2-1}$ and $\overline{\mathcal{E}}^\infty=(\mathcal{E}_{h_1}^\infty)_{0\le h_1\le \iota_1-1}$ represent good coverings (see Definition~\ref{goodcovering}).

Different path deformations performed on the analytic solutions give rise to Theorem~\ref{teo2a} and Theorem~\ref{teo2b}, where upper bounds on the difference of two consecutive solutions are attained (consecutive solutions in the sense that they are related to consecutive sectors in a good covering). Such bounds are related to null Gevrey and $q-$Gevrey asymptotic expansions of some positive order. As a matter of fact, the previous differences allow to apply a novel $((q,k);s)-$version of the cohomological criteria known as Ramis-Sibuya theorem. Such result is related to functions admitting $q-$Gevrey asymptotic expansions of order $k$ and a Gevrey sub-level of order $s$, see Theorem~\ref{teo553}. We also apply a $q-$analog of Ramis-Sibuya Theorem, see Theorem~\ref{teoqrs}.

The main two results of the present work are Theorem~\ref{teopral1} and Theorem~\ref{teopral2} relating the analytic solutions of (\ref{epral}) to their formal power series expansions obtaining asymptotic results of different nature. Such solutions are known as inner and outer solutions (see Definition~\ref{defi7a} and Definition~\ref{defi7b}, resp.). Such asymptotic solutions have also been observed in the previous study~\cite{family3}, in the framework of singularly perturbed PDEs. However, the different nature of the asymptotic expansions regarding the outer and inner solutions is a novel phenomena which has firstly been observed in the present study.

The inner and outer expansions appear in the study of matched asymptotic expansions (see~\cite{om,skinner}, among others for the classical theory). In the work~\cite{frsch} by A. Fruchard and R. Sch\"afke, the method of matching is developed, studying the nature of such asymptotic expansions, under Gevrey settings.

We fix a good covering $(\mathcal{E}^{\infty}_{h_1})_{0\le h_1\le\iota_1-1}$, and consider the holomorphic solutions of the problem (\ref{epral}), $u_{h_1}(\bt,z,\epsilon)$ defined on $\mathcal{E}^{\infty}_{h_1}$ w.r.t. $\epsilon$ for all $0\le h_1\le \iota_1-1$. In Theorem~\ref{teopral1}, we prove that for some $\mu_{2}>\lambda_{2}$, $\theta_{h_1}$ and some adequate domain
$\chi_{2}^{\infty} \subset \mathcal{T}_{2}$, the function
\begin{equation}
\epsilon \mapsto u_{h_1}(t_{1},\frac{x_{2}}{\epsilon^{\mu_{2}}} e^{\sqrt{-1}\theta_{h_1}},z,\epsilon) \label{e117}
\end{equation}
with values in the Banach space of holomorphic and bounded functions on $\mathcal{T}_{1} \times \chi_{2}^{\infty} \times H_{\beta'}$, say
$\mathbb{F}_{1}$, admits a formal power series $\hat{u}^{\infty}(\epsilon) \in \mathbb{F}_{1}[[\epsilon]]$ as $((q,K);S)-$Gevrey asymptotic
expansion on $\mathcal{E}_{h_1}^{\infty}$ for some $K,S>0$, for every $0 \leq h_{1} \leq \iota_{1}-1$. 

The proof of this result leans on the application of accurate estimates related to the  $-1$-branch of Lambert $W$ function (see Lemma~\ref{lema713}).

Concerning the \textit{outer solutions} of the main problem under study, we consider a good covering $(\mathcal{E}^{0}_{h_2})_{0\le h_2\le\iota_2-1}$, and the holomorphic solutions of (\ref{epral}), $u_{h_2}(\bt,z,\epsilon)$ defined on $\mathcal{E}^{0}_{h_2}$ w.r.t. $\epsilon$ for all $0\le h_2\le \iota_2-1$. In Theorem~\ref{teopral2}, we prove that 
\begin{equation}\label{e117b}
\epsilon\mapsto u_{h_2}(\bt,z,\epsilon),\quad 0\le h_2\le \iota_2-1,
\end{equation}
is an outer solution of (\ref{epral}) with values in the Banach space $\mathbb{F}_2$ of holomorphic and bounded functions on $\mathcal{T}_1\times(\mathcal{T}_2\cap D(0,\rho_2))\times H_{\beta'}$. Moreover, there exists a formal power series $\hat{u}^0(\epsilon)\in\mathbb{F}_2[[\epsilon]]$ which is the common $q-$Gevrey asymptotic expansion of some positive order of each solution (\ref{e117b}) on $\mathcal{E}_{h_2}^0$, for $0\le h_2\le\iota_2-1$.

In recent years, an increasing interest on the study of the asymptotic behavior of solutions to $q-$difference-differential equations in the complex domain has been observed. New theories giving rise to $q-$analogs of the classical theory of Borel-Laplace summability have been discussed and studied, as in the case of the work~\cite{taq}, by H. Tahara, where the author also provides information about $q-$analogs of Borel and Laplace transforms and related properties on convolution or Watson-type results. The use of procedures based on the Newton polygon is also exploited in recent studies, such as the work~\cite{taya} by H. Tahara and H. Yamazawa. Also, it is worth mentioning the study of $q-$analogs of Briot-Bouquet type partial differential equations by H. Yamazawa, in~\cite{ya}. Integral transforms involving special functions have also been considered in the study of $q-$difference-differential equations in~\cite{ho,pr}. Other references in this context by the authors and collaborators are listed in the references.

Different kinds of Advanced/delayed partial differential equations are the cornerstone of mathematical models which have been recently studied. Examples of such studies have been applied to tsunamis and rogue waves which can be found in~\cite{pr3}. We also refer to other studies such as~\cite{pr1,pr2}, and the references therein.

The outline of the work is as follows. In Section~\ref{secdos}, we recall some known facts about formal $q-$Borel transform, analytic $q-$Laplace transform and inverse Fourier transform together with some properties which are applied to transform the main equation under study into auxiliary problems. Afterwards, we provide the definition and related properties of some Banach spaces involved in the construction of the solution. In Section~\ref{sec2} we state the main problem under study (\ref{e1}) and two auxiliary equations. The elements involved in them in addition to the domains of existence and upper bounds of the solutions of such equations are detailed. Section~\ref{sec4} is devoted to the existence and description of the domain of existence for the auxiliary equation (\ref{e2}), and associated estimates. In the following section, Section~\ref{sec5}, we provide analytic solutions of (\ref{e1}) (see Theorem~\ref{teo1}) and estimates on the difference of two of them (see Theorem~\ref{teo2a} and Theorem~\ref{teo2b}). After a brief summary on $q-$asymptotic expansions in the first part of Section~\ref{sec55}, we provide formal power series expansions in the perturbation parameter of the analytic solutions and relate them asymptotically in adequate domains. These results are attained in Theorem~\ref{teopral1} and Theorem~\ref{teopral2}. The work concludes with two technical sections, Section~\ref{secanexo2} and Section~\ref{secanexo}, left to the end of the work in order not to interfere with our reasonings.

\section{Review on certain integral operators. Study of Banach spaces involved in the problem}\label{secdos}

We briefly describe the foundations of the analytic and formal operators that will allow the transformation of the main problem under study in terms of auxiliary equations. The solutions of such equations belong to certain Banach spaces which are constructed subsequently.

\subsection{Review of some formal and analytic operators}\label{sec32a}

This section is devoted to recall some of the basic facts on formal and analytic transformations corresponding to the $q-$analogs of those appearing in the classical Borel-Laplace summability theory. These tools were developed in~\cite{dreyfus,ramis}.

Through the whole section, $q>1$ stands for a real number, and $k\ge1$ is a positive integer. $\mathbb{E}$ stands for a complex Banach space.

\begin{defin}
Given $\hat{f}(T)=\sum_{n\ge0}a_nT^n\in\mathbb{E}[[T]]$, the formal $q-$Borel transform of order $k$ of $\hat{f}(T)$ is defined by
$$\hat{\mathcal{B}}_{q;1/k}(\hat{f}(T))(\tau)=\sum_{n\ge0}a_n\frac{\tau^n}{(q^{1/k})^{n(n-1)/2}}\in\mathbb{E}[[\tau]].$$
\end{defin}

The next result, whose proof can be found in Proposition 5~\cite{maq}, is crucial in order to transform the main equation into an auxiliary one lying in the $q-$Borel plane.

\begin{prop}
Let $m\in\N$ and $j\in\mathbb{Q}$. For every $\hat{f}(T)\in\mathbb{E}[[T]]$, it holds that
$$\hat{\mathcal{B}}_{q;1/k}(T^m\sigma_{q;T}^j\hat{f}(T))(\tau)=\frac{\tau^m}{(q^{1/k})^{m(m-1)/2}}\sigma_{q;\tau}^{j-\frac{m}{k}}\left(\hat{\mathcal{B}}_{q;1/k}(\hat{f}(T))(\tau)\right).$$
\end{prop}

The $q-$analog of Laplace transform used in this work was introduced in~\cite{viziozhang}, and makes use of a kernel given by the Jacobi Theta function of order $k$, given by
$$\Theta_{q^{1/k}}(x)=\sum_{n\in\mathbb{Z}}q^{-\frac{n(n-1)}{2k}}x^n,\quad x\in\C^{\star}.$$
We recall that Jacobi Theta function satisfies
$$\Theta_{q^{1/k}}\left(q^{\frac{m}{k}}x\right)=q^{\frac{m(m+1)}{2k}}x^m\Theta_{q^{1/k}}(x),\quad m\in\Z,\quad x\in\C^\star.$$
From Lemma 4.1~\cite{lamaq2}, given any $\tilde{\delta}>0$, there exists $C_{k,q}>0$ (independent of $\tilde{\delta}$) such that
\begin{equation}\label{e190}
\left|\Theta_{q^{1/k}}(x)\right|\ge C_{q,k}\tilde{\delta}\exp\left(\frac{k}{2}\frac{\log^2|x|}{\log(q)}\right)|x|^{1/2}, 
\end{equation}
for all $x\in\C^{\star}$ such that $|1+xq^{\frac{m}{k}}|>\tilde{\delta}$, for every $m\in\Z$. This last property allows to define a $q-$analog of Laplace transform with appropriate properties.

\begin{defin}\label{def120}
Let $\rho>0$ and $S_d$ be an unbounded sector with vertex at 0, and bisecting direction $d\in\R$. Let $f:S_d\cup D(0,\rho)\to\mathbb{E}$ be a holomorphic function, continuous up to the boundary, such that there exist $K>0$, and $\alpha\in\R$ with
$$\left\|f(\tau)\right\|_{\mathbb{E}}\le K\exp\left(\frac{k}{2}\frac{\log^2|\tau|}{\log(q)}+\alpha\log|\tau|\right),\quad \tau\in S_d,|\tau|\ge \rho,$$
$$\left\|f(\tau)\right\|_{\mathbb{E}}\le K,\quad \tau\in\overline{D}(0,\rho).$$
We choose an argument $\gamma\in\R$ within the set of arguments in $S_d$ and define the $q-$Laplace transform of order $k$ of $f$ in direction $\gamma$ by
$$\mathcal{L}^{\gamma}_{q;1/k}(f(\tau))(T)=\frac{1}{\pi_{q^{1/k}}}\int_{L_{\gamma}}\frac{f(u)}{\Theta_{q^{1/k}}\left(\frac{u}{T}\right)}\frac{du}{u},$$
where $L_{\gamma}=\{re^{\gamma\sqrt{-1}}:r\in(0,\infty)\}$, and $\pi_{q^{1/k}}:=\frac{\log(q)}{k}\prod_{n\ge 0}(1-q^{-\frac{n+1}{k}})^{-1}$.
\end{defin}

The proof of the next results can be found in detail in Lemma 4 and Proposition 6~\cite{maq}.

\begin{lemma}
Let $\tilde{\delta}>0$. In the situation of Definition~\ref{def120}, the integral transform $\mathcal{L}^{\gamma}_{q;1/k}(f(\tau))(T)$ defines a bounded holomorphic function on the domain $\mathcal{R}_{\gamma,\tilde{\delta}}\cap D(0,r_1)$ for every $0<r_1\le q^{(1/2-\alpha)/k}/2$, where
$$\mathcal{R}_{\gamma,\tilde{\delta}}=\left\{T\in\C^{\star}:\left|1+\frac{re^{\gamma\sqrt{-1}}}{T}\right|>\tilde{\delta},\hbox{ for all }r\ge 0\right\}.$$
The value of $\mathcal{L}^{\gamma}_{q;1/k}(f(\tau))(T)$ does not depend on the choice of $\gamma$ under $e^{\gamma\sqrt{-1}}\in S_d$. 
\end{lemma}

\begin{prop}\label{prop6}
Let $f$ be as in Definition~\ref{def120}, and let $\tilde{\delta}>0$. Then, for all $\sigma\ge 0$ one has
$$T^{\sigma}\sigma_q^{j}(\mathcal{L}^{\gamma}_{q;1/k}f(\tau))(T)=\mathcal{L}_{q;1/k}^{\gamma}\left(\frac{\tau^\sigma}{(q^{1/k})^{\sigma(\sigma-1)/2}}\sigma_q^{j-\frac{\sigma}{k}}f(\tau)\right)(T),$$
for all $T\in\mathcal{R}_{\gamma,\tilde{\delta}}\cap D(0,r_1)$, where $0<r_1<q^{(\frac{1}{2}-\alpha)/k}/2$.
\end{prop}

We conclude with the definition and properties regarding the inverse Fourier transform. 

\begin{defin}
Let $\mu,\beta\in\R$. We write $E_{(\beta,\mu)}$ for the vector space of continuous functions $h:\R\to\C$ such that 
$$\left\|h(m)\right\|_{(\beta,\mu)}=\sup_{m\in\R}(1+|m|)^{\mu}\exp(\beta|m|)|h(m)|<\infty.$$
The pair $(E_{(\beta,\mu)},\left\|\cdot\right\|_{(\beta,\mu)})$ turns out to be a Banach space.
\end{defin}

\begin{prop}
Let $\mu>1,\beta>0$ and $f\in E_{(\beta,\mu)}$. The inverse Fourier transform of $f$ is defined by
$$\mathcal{F}^{-1}(f)(x)=\frac{1}{(2\pi)^{1/2}}\int_{-\infty}^{\infty}f(m)\exp(ixm)dm,\quad x\in\R.$$
This function can be extended to an analytic function on the horizontal strip $H_\beta=\{z\in\C:|\hbox{Im}(z)|<\beta\}$. Moreover, it holds that the function $\phi(m)=imf(m)$ belongs to $E_{(\beta,\mu-1)}$ and
$$\partial_z\mathcal{F}^{-1}(f)(z)=\mathcal{F}^{-1}(\phi)(z),\quad z\in H_\beta.$$
In addition to this, it holds that the convolution product of $f\in E_{(\beta,\mu)}$ and $g\in E_{(\beta,\mu)}$, defined by
$$\psi(m)=\frac{1}{(2\pi)^{1/2}}\int_{-\infty}^{\infty}f(m-m_1)g(m_1)dm_1$$
is such that $\psi\in E_{(\beta,m)}$, and $\mathcal{F}^{-1}(f)(z)\mathcal{F}^{-1}(g)(z)=\mathcal{F}^{-1}(\psi)(z),$ for every $z\in H_\beta$.
\end{prop}

\subsection{Banach spaces of functions of $q-$exponential growth and exponential decay}\label{sec33}

In this section, we state the definition of the complex Banach space $qExp_{(k,\beta,\mu,\alpha)}^{d}$. The analytic solution of the main equation under study is built departing from one element in such Banach space via $q-$Laplace transform and inverse Fourier transform. Similar versions of this Banach space have already appeared in previous works by the authors such as~\cite{lamaq2,lamaq}, and the contribution~\cite{maq} of the second author.

In the whole section, we fix real numbers $\beta,\mu,k>0$, $q,\delta>1$ and $\alpha$. We also set $d\in\R$ and choose  an infinite sector $S_{d}$ of bisecting direction $d$, with vertex at the origin; and the closed disc $\overline{D}(0,\rho)$, for some $\rho>0$. 

\begin{defin}\label{defi11}
Let $qExp_{(k,\beta,\mu,\alpha)}^{d}$ be the vector space of all continuous complex valued functions $(\tau,m)\mapsto f(\tau,m)$ on $(S_d\cup \overline{D}(0,\rho))\times\R$ which are holomorphic with respect to the first variable on $S_d\cup D(0,\rho)$ and such that
$$\left\|f(\tau,m)\right\|_{(k,\beta,\mu,\alpha)}:=\sup_{\substack{\tau\in S_d\cup \overline{D}(0,\rho)\\m\in\R}}(1+|m|)^{\mu}e^{\beta|m|}\exp\left(\frac{-k\log^2(|\tau|+\delta)}{2\log(q)}-\alpha\log(|\tau|+\delta)\right)|f(\tau,m)|$$
is finite. The pair $(qExp_{(k,\beta,\mu,\alpha)}^{d},\left\|\cdot\right\|_{(k,\beta,\mu,\alpha)})$ is a Banach space.
\end{defin}

In the forthcoming results, we describe some properties on the elements of the previous Banach space when combined with certain operators. The first result is a direct consequence of the previous definition.

\begin{lemma}\label{lema1}
Let $(\tau,m)\mapsto a(\tau,m)$ be a bounded continuous function defined on $(S_d\cup \overline{D}(0,\rho))\times\R$, holomorphic with respect to $\tau$ on the set $S_d\cup D(0,\rho)$. For every $f(\tau,m)\in qExp_{(k,\beta,\mu,\alpha)}^{d}$ one has that $a(\tau,m)f(\tau,m)$ belongs to $qExp_{(k,\beta,\mu,\alpha)}^{d}$, and it holds that 
$$ \left\|a(\tau,m)f(\tau,m)\right\|_{(k,\beta,\mu,\alpha)}\le C_1 \left\|f(\tau,m)\right\|_{(k,\beta,\mu,\alpha)},$$
with $C_1:=\sup_{\substack{\tau\in S_d\cup\overline{D}(0,\rho)\\m\in\R}}|a(\tau,m)|$.
\end{lemma}

\begin{prop}\label{prop1}
Let $\gamma_1,\gamma_2,\gamma_3\ge0$ such that
\begin{equation}\label{e136}
\gamma_2\le k\gamma_3+\gamma_1.
\end{equation}
Let $a_{\gamma_1}(\tau)$ be a bounded holomorphic function defined on $S_d\cup D(0,\rho)$ with $(1+|\tau|)^{\gamma_1}|a_{\gamma_1}(\tau)|\le 1$ for every $\tau\in S_d\cup D(0,\rho)$. Then, for all $f\in qExp_{(k,\beta,\mu,\alpha)}^{d}$ it holds that 
\begin{equation}\label{e140}
\left\|a_{\gamma_1}(\tau)\tau^{\gamma_2}\sigma_{q;\tau}^{-\gamma_3}f(\tau,m)\right\|_{(k,\beta,\mu,\alpha)}\le C_2\left\|f(\tau,m)\right\|_{(k,\beta,\mu,\alpha)},
\end{equation}
for some $C_2>0$.
\end{prop}
\begin{proof}
Let $f\in qExp_{(k,\beta,\mu,\alpha)}^{d}$. Then, it holds that
\begin{multline}
\left\|a_{\gamma_1}(\tau)\tau^{\gamma_2}\sigma_{q;\tau}^{-\gamma_3}f(\tau,m)\right\|_{(k,\beta,\mu,\alpha)}\\
=\sup_{\substack{\tau\in S_d\cup \overline{D}(0,\rho)\\m\in\R}}C_{21}(|\tau|)(1+|m|)^{\mu}e^{\beta|m|}\exp\left(\frac{-k\log^2(\frac{|\tau|}{q^{\gamma_3}}+\delta)}{2\log(q)}-\alpha\log\left(\frac{|\tau|}{q^{\gamma_3}}+\delta\right)\right)\left|f\left(\frac{\tau}{q^{\gamma_3}},m\right)\right|\\
\le \left\|f(\tau,m)\right\|_{(k,\beta,\mu,\alpha)}\sup_{\tau\in S_d\cup \overline{D}(0,\rho)}C_{21}(|\tau|)
\end{multline}
with 
\begin{multline}
C_{21}(|\tau|)= \frac{|\tau|^{\gamma_2}}{(1+|\tau|)^{\gamma_1}}\exp\left(\frac{-k\left(\log^2(|\tau|+\delta)-\log^2(\frac{|\tau|}{q^{\gamma_3}}+\delta)\right)}{2\log(q)}\right)\\
\times\exp\left(-\alpha\log(|\tau|+\delta)+\alpha\log\left(\frac{|\tau|}{q^{\gamma_3}}+\delta\right)\right).
\end{multline}
Observe that $\sigma_{q;\tau}^{-\gamma_3}(S_d\cup \overline{D}(0,\rho))\subseteq S_d\cup \overline{D}(0,\rho)$. In order to give upper bounds for $C_{21}(|\tau|)$, we have
\begin{align*}
\log^2&(\frac{|\tau|}{q^{\gamma_3}}+\delta)-\log^2(|\tau|+\delta)\\
&=\log^2(|\tau|+\delta q^{\gamma_3})-\log^2(|\tau|+\delta)+(\log(q)\gamma_3)^2-2\gamma_3\log(q)\log(|\tau|+\delta q^{\gamma_3}).
\end{align*}
The mean value theorem guarantees that 
$$\log^2(|\tau|+\delta q^{\gamma_3})-\log^2(|\tau|+\delta)\le 2\delta(q^{\gamma_3}-1)\max_{x\in [|\tau|+\delta,|\tau|+\delta q^{\gamma_3}]}\frac{\log(x)}{x}\le C_{22},$$
for some positive constant $C_{22}$ which does not depend on $\tau\in S_d\cup \overline{D}(0,\rho)$.
This entails that
$$C_{21}(|\tau|)\le \exp\left(\frac{kC_{22}}{2\log(q)}+(\log(q)\gamma_3)^2\right)\frac{|\tau|^{\gamma_2}}{(1+|\tau|)^{\gamma_1}}\left(\frac{|\tau|/q^{\gamma_3}+\delta}{|\tau|+\delta}\right)^{\alpha}\frac{1}{(|\tau|+\delta q^{\gamma_3})^{k\gamma_3}},$$
for every $\tau\in S_d\cup \overline{D}(0,\rho)$. We conclude the result by taking into account (\ref{e136}). 
\end{proof}

The next result is stated without proof, heavily rests on Proposition 2~\cite{lama19}, and it is based on bounds stated in Lemma 2.2~\cite{costintanveer}.

\begin{prop}\label{prop235}
Let $R_1,R_2\in\C[X]$ with $\hbox{deg}(R_1)\ge \hbox{deg}(R_2)$, and such that $R_1(im)\neq 0$ for all $m\in\R$. Let $\mu>\hbox{deg}(R_2)+1$. Given $f\in E_{(\beta,\mu)}$ and $g\in q\hbox{Exp}^d_{(k,\beta,\mu,\alpha)}$, then the function
$$\Phi(\tau,m):=\frac{1}{R_1(im)}\int_{-\infty}^{\infty}f(m-m_1)R_2(im_1)g(\tau,m_1)dm_1$$
belongs to $q\hbox{Exp}^d_{(k,\beta,\mu,\alpha)}$, and there exists $C'_2>0$ such that
$$\left\|\Phi(\tau,m)\right\|_{(k,\beta,\mu,\alpha)}\le C'_2\left\|f(m)\right\|_{(\beta,\mu)}\left\|g(\tau,m)\right\|_{(k,\beta,\mu,\alpha)}.$$
\end{prop}

\section{Statement of the main problem and auxiliary equations}\label{sec2}

Let $\beta>0$ and $k'_1,k_1,k_2\ge 1$ with $k'_1>k_1$, and $D_1,D_2\ge 2$ be integer numbers. We also consider a real number $q>1$, and choose real numbers $\Delta_{D_1D_2},d_{D_1},\tilde{\delta}_{D_2}\ge0$, and $\Delta_{\ell_1\ell_2},\delta_{\ell_1},d_{\ell_1},\tilde{\delta}_{\ell_2}\ge0$ for every $1\le \ell_1\le D_1-1$ and $1\le \ell_2\le D_2-1$. We assume that
\begin{equation}\label{e93}
\tilde{\delta}_{\ell_2}\le \left(\frac{k'_1}{k_1}-1\right)d_{\ell_1}+d_{D_1}+\tilde{\delta}_{D_2}-k'_1\delta_{\ell_1},\quad 1\le \ell_1\le D_1-1,1\le \ell_2\le D_2-1.
\end{equation}
In addition to that, there exist natural numbers $\lambda_1,\lambda_2$ with
\begin{equation}\label{e224}
\Delta_{D_1D_2}=\lambda_1 d_{D_1}+\lambda_2 k_2\tilde{\delta}_{D_2},
\end{equation}
\begin{equation}\label{e226}
k_1\delta_{\ell_1}<d_{\ell_1},\quad  \lambda_1d_{\ell_1}+\lambda_2k_2\tilde{\delta}_{\ell_2}< \Delta_{\ell_1\ell_2},\quad 1\le \ell_1\le D_1-1,1\le \ell_2\le D_2-1.
\end{equation}

Let us also fix polynomials $Q,R_{D_1D_2}$, and $R_{\ell_1\ell_2}$ for all $1\le \ell_1\le D_1-1$ and $1\le \ell_2\le D_2-1$, with complex coefficients such that
\begin{equation}\label{e97}
\frac{Q(im)}{R_{D_1D_2}(im)}\in S_{Q,R_{D_1D_2}},\quad m\in\R,
\end{equation}
where 
$$S_{Q,R_{D_1D_2}}:=\left\{z\in\C:|z|\ge r_{Q,R_{D_1D_2}} ,|\arg(z)-d_{Q,R_{D_1D_2}}|\le \eta_{Q,R_{D_1D_2}}  \right\},$$
for some $r_{Q,R_{D_1D_2}},\eta_{Q,R_{D_1D_2}}>0$ and $d_{Q,R_{D_1D_2}}\in\R$. Observe this condition implies $\hbox{deg}(R_{D_1D_2})\le \hbox{deg}(Q)$. These polynomials are chosen in such a way that $R_{D_1D_2}(im)\neq0$ for $m\in\R$, and
\begin{equation}\label{e112}
\hbox{deg}(R_{D_1D_2})\ge \deg(R_{\ell_1\ell_2}),\quad 1\le \ell_1\le D_1-1,1\le \ell_2\le D_2-1.
\end{equation}

Let $\mu>1$ such that
\begin{equation}\label{e270}
\mu>\hbox{deg}(R_{\ell_1\ell_2})+1,\quad 1\le \ell_1\le D_1-1,1\le \ell_2\le D_2-1.
\end{equation}

The main problem under study in this work is 
\begin{multline}\label{e1}
Q(\partial_z)u(\bt,z,\epsilon)=\epsilon^{\Delta_{D_1D_2}}(t_2^{k_2+1}\partial_{t_2})^{\tilde{\delta}_{D_2}}t_1^{d_{D_1}}\sigma_{q;t_1}^{\frac{d_{D_1}}{k_1}}R_{D_1D_2}(\partial_z)u(\bt,z,\epsilon)\\
+\sum_{\substack{1\le\ell_1\le D_1-1\\1\le\ell_2\le D_2-1}}\epsilon^{\Delta_{\ell_1\ell_2}}t_1^{d_{\ell_1}}\sigma_{q;t_1}^{\delta_{\ell_1}}\sigma_{q;t_2}^{\frac{1}{k_2}\left(\delta_{\ell_1}-\frac{d_{\ell_1}}{k_1}\right)}(t_2^{k_2+1}\partial_{t_2})^{\tilde{\delta}_{\ell_2}}c_{\ell_1\ell_2}(z,\epsilon)R_{\ell_1\ell_2}(\partial_z)u(\bt,z,\epsilon)+f(\bt,z,\epsilon),
\end{multline}
under initial data $u(t_1,0,z,\epsilon)\equiv u(0,t_2,z,\epsilon)\equiv 0$.

The function $f$ is a holomorphic function in $\C^{\star}\times\C^{\star}\times H_{\beta'}\times (D(0,\epsilon_0)\setminus\{0\})$ for every $0<\beta'<\beta$, where $H_{\beta'}$ stands for the horizontal strip $H_{\beta'}=\{z\in\C:|\hbox{Im}(z)|< \beta'\}$. It is constructed as follows. Let $\psi(\tau,m,\epsilon)$ be a continuous function, continuous in $\C\times\R\times D(0,\epsilon_0)$, for some $\epsilon_0>0$, entire with respect to the first variable, and holomorphic with respect to the third one in $D(0,\epsilon_0)$. We moreover assume there exists $C_\psi>0$ such that
\begin{equation}\label{epsi}
|\psi(\tau,m,\epsilon)|\le C_\psi (1+|m|)^{-\mu}e^{-\beta|m|}\exp\left(\frac{k'_1}{2\log(q)}\log^2(|\tau|+\delta)+\alpha\log(|\tau|+\delta)\right),
\end{equation} 
for every $(\tau,m,\epsilon)\in\C\times\R\times D(0,\epsilon_0)$, some $\alpha\in\R$, $\delta>0$; and where $\mu$ satisfies (\ref{e270}).

In view of the definition of $q-$Laplace transform and the results described in Section~\ref{sec32a}, one can define
$$F(\bT,m,\epsilon)=\frac{1}{\pi_q^{1/k_1}}\int_{L_\gamma}\psi(u,m,\epsilon)\frac{1}{\Theta_{q^{1/k_1}}\left(\frac{u}{T_1}\right)}e^{-\left(\frac{1}{T_2}\right)^{k_2}u}\frac{du}{u},$$
for some fixed $\gamma\in\R$. Here, $\bT:=(T_1,T_2)$. The function 
\begin{equation}\label{e243}
f( \bt,z,\epsilon)=\mathcal{F}^{-1}\left(m\mapsto F(\epsilon^{\lambda_1}t_1,\epsilon^{\lambda_2}t_2,m,\epsilon)\right)(z),
\end{equation}
turns out to be holomorphic on the set $\C^\star\times\C^{\star}\times H_{\beta'}\times (D(0,\epsilon_0)\setminus\{0\})$. Observe that in the previous construction, given $T_1,T_2\in\C^{\star}$, one can choose $\gamma\in\R$ with $\cos(\gamma-k_2\hbox{arg}(T_2))>0$ and $|1+(uq^{m/k})/T_1|>0$ for all $m\in\Z$ and with $u\in L_\gamma$. 

For every $0\le \ell_1\le D_1-1$ and $0\le \ell_2\le D_2-1$, the function $c_{\ell_1\ell_2}(z,\epsilon)$ is constructed in the following way:
$$c_{\ell_1\ell_2}(z,\epsilon)=\frac{1}{(2\pi)^{1/2}}\int_{-\infty}^{\infty}C_{\ell_1\ell_2}(m,\epsilon)e^{izm}dm,\quad (z,\epsilon)\in H_{\beta}\times D(0,\epsilon_0).$$
$c_{\ell_1\ell_2}(z,\epsilon)$ turns out to be a holomorphic function on $H_\beta\times D(0,\epsilon_0)$ whenever $m\mapsto C_{\ell_1\ell_2}(m,\epsilon)\in E_{(\beta,\mu)}$. In addition to that, we assume that uniform bounds with respect to the perturbation parameter are satisfied, i.e. there exist $\mathcal{C}_{\ell_1\ell_2}>0$ with
\begin{equation}\label{e284}
\sup_{\epsilon\in D(0,\epsilon_0)}\left\|C_{\ell_1\ell_2}(m,\epsilon)\right\|_{(\beta,\mu)}\le\mathcal{C}_{\ell_1\ell_2}.
\end{equation}

\subsection{Study of auxiliary equations}\label{sec31}

In this section we preserve the statements and constructions concerning the main problem under study (\ref{e1}), and the geometric and algebraic conditions held on the elements involved in the main equation.

We search for solutions of (\ref{e1}) of the form 
\begin{equation}\label{e225}
u(\bt,z,\epsilon)=\mathcal{F}^{-1}(m\mapsto U(\epsilon^{\lambda_1}t_1,\epsilon^{\lambda_2}t_2,m,\epsilon))(z).
\end{equation}
Assuming the solution is of the form (\ref{e225}), the expression $U(T_1,T_2,m,\epsilon)$ solves 
\begin{multline}\label{e2}
Q(im)U(\bT,m,\epsilon)=T_1^{d_{D_1}}\sigma_{q;T_1}^{\frac{d_{D_1}}{k_1}}(T_2^{k_2+1}\partial_{T_2})^{\tilde{\delta}_{D_2}}R_{D_1D_2}(im)U(\bT,m,\epsilon)\\
+\sum_{\substack{1\le\ell_1\le D_1-1\\1\le\ell_2\le D_2-1}}\epsilon^{\Delta_{\ell_1\ell_2}-\lambda_1 d_{\ell_1}-\lambda_2 k_2\tilde{\delta}_{\ell_2}}T_1^{d_{\ell_1}}\sigma_{q;T_1}^{\delta_{\ell_1}}\sigma_{q;T_2}^{\frac{1}{k_2}(\delta_{\ell_1}-\frac{d_{\ell_1}}{k_1})}(T_2^{k_2+1}\partial_{T_2})^{\tilde{\delta}_{\ell_2}}\\
\times \frac{1}{(2\pi)^{1/2}}\int_{-\infty}^{\infty}C_{\ell_1\ell_2}(m-m_1,\epsilon)R_{\ell_1\ell_2}(im_1)U(\bT,m_1,\epsilon)dm_1+F(\bT,m,\epsilon).
\end{multline}

We reduce the study of solutions of (\ref{e1}) to those of (\ref{e2}), which are linked through (\ref{e225}). In order to solve (\ref{e2}), we adapt a recent approach developed in ~\cite{family3} to a new situation involving both partial differential and $q-$difference operators. We seek for solutions of (\ref{e2}) of the special form
\begin{equation}\label{e242}
U_{\gamma}(\bT,m,\epsilon)=\frac{1}{\pi_{q^{1/k_1}}}\int_{L_\gamma}\omega(u,m,\epsilon)\frac{1}{\Theta_{q^{1/k_1}}\left(\frac{u}{T_1}\right)}e^{-\left(\frac{1}{T_2}\right)^{k_2}u}\frac{du}{u},
\end{equation} 
for some appropriate function $\omega(\tau,m,\epsilon)$ and $\gamma\in\R$. We refer to Section~\ref{sec32a} for the definitions of the elements involved in the previous expression. 

Let us consider a second auxiliary equation: 
\begin{multline}\label{e3}
Q(im)\omega(\tau,m,\epsilon)=\frac{k_2^{\tilde{\delta}_{D_2}}}{(q^{1/k_1})^{d_{D_1}(d_{D_1}-1)/2}}\tau^{d_{D_1}+\tilde{\delta}_{D_2}}R_{D_1D_2}(im)\omega(\tau,m,\epsilon)\\
+\sum_{\substack{1\le\ell_1\le D_1-1\\1\le\ell_2\le D_2-1}}\frac{\epsilon^{\Delta_{\ell_1\ell_2}-\lambda_1d_{\ell_1}-\lambda_2 k_2 \tilde{\delta}_{\ell_2}}k_2^{\tilde{\delta}_{\ell_2}}q^{(\delta_{\ell_1}-d_{\ell_1}/k_1)\tilde{\delta}_{\ell_2}}}{(q^{1/k_1})^{d_{\ell_1}(d_{\ell_1}-1)/2}}\tau^{\tilde{\delta}_{\ell_2}+d_{\ell_1}}\\
\times \frac{1}{(2\pi)^{1/2}}\int_{-\infty}^{\infty}C_{\ell_1\ell_2}(m-m_1,\epsilon)R_{\ell_1\ell_2}(im_1)\omega(q^{\delta_{\ell_1}-\frac{d_{\ell_1}}{k_1}}\tau,m_1,\epsilon)dm_1+\psi(\tau,m,\epsilon).
\end{multline}

We define the polynomial $P_m(\tau)$ by
\begin{equation}\label{e261}
P_m(\tau)=Q(im)-\frac{k_2^{\tilde{\delta}_{D_2}}}{(q^{1/k_1})^{d_{D_1}(d_{D_1}-1)/2}}\tau^{d_{D_1}+\tilde{\delta}_{D_2}}R_{D_1D_2}(im),
\end{equation}
whose factorization is given by
$$P_m(\tau)=-\frac{k_2^{\tilde{\delta}_{D_2}}R_{D_1D_2}(im)}{(q^{1/k_1})^{d_{D_1}(d_{D_1}-1)/2}}\prod_{\ell=0}^{d_{D_1}+\tilde{\delta}_{D_2}-1}(\tau-q_{\ell}(m)),$$
with
$$q_\ell(m)=\Delta_m\exp\left(\sqrt{-1}\left(\frac{1}{d_{D_1}+\tilde{\delta}_{D_2}}\arg\left(\frac{Q(im)(q^{1/k_1})^{d_{D_1}(d_{D_1}-1)/2}}{k_2^{\tilde{\delta}_{D_2}}R_{D_1D_2}(im)}\right)+\frac{2\pi\ell}{d_{D_1}+\tilde{\delta}_{D_2}}\right)\right),$$
for all $0\le \ell\le d_{D_1}+\tilde{\delta}_{D_2}-1$, where $\Delta_m=\left(\frac{|Q(im)|}{k_2^{\tilde{\delta}_{D_2}}|R_{D_1D_2}(im)|}(q^{1/k_1})^{\frac{d_{D_1}(d_{D_1}-1)}{2}}\right)^{1/(d_{D_1}+\tilde{\delta}_{D_2})}$.

Let $d\in\R$ be such that the infinite sector $S_d$ of bisecting direction $d$ satisfies the following geometric construction:\label{geosd}
there exists $M_1>0$ such that $|\tau-q_\ell(m)|\ge M_1(1+|\tau|)$, for all $0\le \ell\le d_{D_1}+\tilde{\delta}_{D_2}-1$, $m\in\R$ and $\tau\in S_d\cup\overline{D}(0,\rho)$. The previous is a feasible condition for an appropriate choice of small enough $\eta_{Q,R_{D_1D_2}}>0$ and large enough $r_{Q,R_{D_1D_2}}>0$, in view of (\ref{e97}) and the definition of $q_\ell(m)$. More precisely, one chooses $S_d$ such that $q_\ell(m)/\tau$ has positive distance to 1, for all $0\le \ell\le d_{D_1}+\tilde{\delta}_{D_2}-1$, $m\in\R$, and $\tau\in S_d\cup\overline{D}(0,\rho)$.
The previous choice of $d$ yields
\begin{equation}
|P_m(\tau)|\ge C_{P}|R_{D_1D_2}(im)|(1+|\tau|)^{d_{D_1}+\tilde{\delta}_{D_2}},\label{e276}
\end{equation}
for some $C_P>0$, valid for all $\tau\in S_d\cup \overline{D}(0,\rho)$ and all $m\in\R$.

\begin{prop}\label{prop5}
Let $S_d$ be an infinite sector with vertex at the origin satisfying the previous geometric conditions. Then, if $C_\psi>0$ (see (\ref{epsi}) for its definition) and $\epsilon_0>0$ are small enough, there exists $\varpi>0$ such that the equation (\ref{e3}) admits a unique solution $\omega^d(\tau,m,\epsilon)$ which belongs to $qExp_{(k'_1,\beta,\mu,\alpha)}^{d}$ with $\left\|\omega^d(\tau,m,\epsilon)\right\|_{(k'_1,\beta,\mu,\alpha)}\le\varpi$, for all $\epsilon\in D(0,\epsilon_0)$.
\end{prop}
\begin{proof}
Let $\epsilon\in D(0,\epsilon_0)$. We consider the map $\mathcal{H}_\epsilon$ defined by
\begin{multline*}
\mathcal{H}_\epsilon(\omega(\tau,m)):=\frac{1}{P_m(\tau)}\left(\sum_{\substack{1\le\ell_1\le D_1-1\\1\le\ell_2\le D_2-1}}\frac{\epsilon^{\Delta_{\ell_1\ell_2}-\lambda_1d_{\ell_1}-\lambda_2 k_2 \tilde{\delta}_{\ell_2}}k_2^{\tilde{\delta}_{\ell_2}}q^{(\delta_{\ell_1}-d_{\ell_1}/k_1)\tilde{\delta}_{\ell_2}}}{(q^{1/k_1})^{d_{\ell_1}(d_{\ell_1}-1)/2}}\tau^{\tilde{\delta}_{\ell_2}+d_{\ell_1}}\right.\\
\left.\times \frac{1}{(2\pi)^{1/2}}\int_{-\infty}^{\infty}C_{\ell_1\ell_2}(m-m_1,\epsilon)R_{\ell_1\ell_2}(im_1)\omega(q^{\delta_{\ell_1}-\frac{d_{\ell_1}}{k_1}}\tau,m_1)dm_1\right)+\frac{1}{P_m(\tau)}\psi(\tau,m,\epsilon).
\end{multline*}

Given $\varpi>0$, let $\omega(\tau,m)\in qExp_{(k'_1,\beta,\mu,\alpha)}^{d}$ with $\left\|\omega^d(\tau,m,\epsilon)\right\|_{(k'_1,\beta,\mu,\alpha)}\le\varpi$. In view of (\ref{e93}), (\ref{e224}), (\ref{e226}), (\ref{e276}) and from Lemma~\ref{lema1}, Proposition~\ref{prop1} and Proposition~\ref{prop235} one has 
\begin{multline}
\left\|\frac{\epsilon^{\Delta_{\ell_1\ell_2}-\lambda_1d_{\ell_1}-\lambda_2 k_2 \tilde{\delta}_{\ell_2}}k_2^{\tilde{\delta}_{\ell_2}}q^{(\delta_{\ell_1}-d_{\ell_1}/k_1)\tilde{\delta}_{\ell_2}}}{(2\pi)^{1/2}(q^{1/k_1})^{d_{\ell_1}(d_{\ell_1}-1)/2}}\frac{\tau^{\tilde{\delta}_{\ell_2}+d_{\ell_1}}}{P_m(\tau)}\right.\\
\times\left.\int_{-\infty}^{\infty}C_{\ell_1\ell_2}(m-m_1,\epsilon)R_{\ell_1\ell_2}(im_1)\omega(q^{\delta_{\ell_1}-\frac{d_{\ell_1}}{k_1}}\tau,m_1)dm_1\right\|_{(k'_1,\beta,\mu,\alpha)}\\
\le \frac{\epsilon_0^{\Delta_{\ell_1\ell_2}-\lambda_1d_{\ell_1}-\lambda_2 k_2 \tilde{\delta}_{\ell_2}}k_2^{\tilde{\delta}_{\ell_2}}q^{(\delta_{\ell_1}-d_{\ell_1}/k_1)\tilde{\delta}_{\ell_2}}}{(2\pi)^{1/2}(q^{1/k_1})^{d_{\ell_1}(d_{\ell_1}-1)/2}C_P}C_2C'_2\mathcal{C}_{\ell_1\ell_2}\left\|\omega(\tau,m,\epsilon)\right\|_{(k'_1,\beta,\mu,\alpha)}\label{e304}
\end{multline}
for every $0\le \ell_1\le D_1-1$ and $0\le \ell_2\le D_2-1$. In addition to this, Lemma~\ref{lema1} and (\ref{e276}) yield
\begin{align}
\left\|\frac{1}{P_m(\tau)}\psi(\tau,m,\epsilon)\right\|_{(k'_1,\beta,\mu,\alpha)}&\le\frac{1}{C_P}\sup_{m\in\R}\frac{1}{|R_{D_1D_2}(im)|}\left\|\psi(\tau,m,\epsilon)\right\|_{(k'_1,\beta,\mu,\alpha)}\nonumber\\
&\le\frac{1}{C_P}\sup_{m\in\R}\frac{1}{|R_{D_1D_2}(im)|}C_{\psi}\label{e309}
\end{align}

Let $\epsilon_0, C_{\psi}>0$ be small enough satisfying
\begin{multline*}
\sum_{\substack{1\le\ell_1\le D_1-1\\1\le\ell_2\le D_2-1}}\frac{\epsilon_0^{\Delta_{\ell_1\ell_2}-\lambda_1d_{\ell_1}-\lambda_2 k_2 \tilde{\delta}_{\ell_2}}k_2^{\tilde{\delta}_{\ell_2}}q^{(\delta_{\ell_1}-d_{\ell_1}/k_1)\tilde{\delta}_{\ell_2}}}{(2\pi)^{1/2}(q^{1/k_1})^{d_{\ell_1}(d_{\ell_1}-1)/2}C_P}C_2C'_2\mathcal{C}_{\ell_1\ell_2}\varpi\\
+\frac{1}{C_P}\sup_{m\in\R}\frac{1}{|R_{D_1D_2}(im)|}C_{\psi}\le\varpi.
\end{multline*}
Then, the estimates (\ref{e304}) and (\ref{e309}) yield to $\left\|\mathcal{H}(\omega(\tau,m)\right\|_{(k'_1,\beta,\mu,\alpha)}\le\varpi$. In other words, the operator $\mathcal{H}_\epsilon$ restricted to $\overline{D}(0,\varpi)\subseteq qExp_{(k'_1,\beta,\mu,\alpha)}^{d}$ is such that $\mathcal{H}_\epsilon(\overline{D}(0,\varpi))\subseteq \overline{D}(0,\varpi)$.

On the other hand, let $\omega_1(\tau,m),\omega_2(\tau,m)\in \overline{D}(0,\rho)\subseteq qExp_{(k'_1,\beta,\mu,\alpha)}^{d}$. Analogously to (\ref{e304}), one arrives at
$$\left\|\mathcal{H}_\epsilon(\omega_1(\tau,m))-\mathcal{H}_\epsilon(\omega_2(\tau,m))\right\|_{(k'_1,\beta,\mu,\alpha)}\le \frac{1}{2}\left\|\omega_1(\tau,m)-\omega_2(\tau,m)\right\|_{(k'_1,\beta,\mu,\alpha)}$$
by choosing $\epsilon_0>0$ such that
$$\sum_{\substack{1\le\ell_1\le D_1-1\\1\le\ell_2\le D_2-1}}\frac{\epsilon_0^{\Delta_{\ell_1\ell_2}-\lambda_1d_{\ell_1}-\lambda_2 k_2 \tilde{\delta}_{\ell_2}}k_2^{\tilde{\delta}_{\ell_2}}q^{(\delta_{\ell_1}-d_{\ell_1}/k_1)\tilde{\delta}_{\ell_2}}}{(2\pi)^{1/2}(q^{1/k_1})^{d_{\ell_1}(d_{\ell_1}-1)/2}C_P}C_2C'_2\mathcal{C}_{\ell_1\ell_2}\le\frac{1}{2}.$$
We conclude that the map $\mathcal{H}_\epsilon:\overline{D}(0,\varpi)\subseteq qExp_{(k'_1,\beta,\mu,\alpha)}^{d}\to qExp_{(k'_1,\beta,\mu,\alpha)}^{d}$ is contractive. The classical fixed point theory in complete metric spaces states the existence of a unique fixed point for $\mathcal{H}_\epsilon$, say $\omega^d(\tau,m,\epsilon)$, in $qExp_{(k'_1,\beta,\mu,\alpha)}^{d}$, with $\left\|\omega^d(\tau,m,\epsilon)\right\|_{(k'_1,\beta,\mu,\alpha)}\le \varpi$. For every $\epsilon\in D(0,\epsilon_0)$, the function $\omega^d(\tau,m,\epsilon)$ is a solution of (\ref{e3}) in view of the definition of the operator $\mathcal{H}_\epsilon$. Holomorphy of the map $D(0,\epsilon_0)\ni\epsilon\mapsto\omega^d(\tau,m,\epsilon)$ is derived from the construction of the fixed point.
\end{proof}

In order to prove that the solutions of (\ref{e2}) and (\ref{e3}) are related via (\ref{e242}), we need to clarify how operators involved in (\ref{e2}) are transformed into the corresponding ones in (\ref{e3}). This is left to the end of the work, in Section~\ref{secanexo} not to interfere our line of reasoning.

\section{Domains of existence for the solutions of (\ref{e2}), and associated estimates}\label{sec4}

In this section, we describe appropriate domains on the time variables in which the solution of the main problem under consideration is well defined, within appropriate \textit{geometric conditions}. Let $\tilde{\mathcal{T}}_1$ be a bounded sector with vertex at the origin such that there exists $\delta_1>0$ with
$$\left|1+\frac{re^{\gamma\sqrt{-1}}}{T_1}\right|\ge \delta_1,$$
for all $r\ge0$ and all $T_1\in\tilde{\mathcal{T}}_1$, and $\gamma$ being an argument in $S_d$. We also fix an unbounded sector $\tilde{\mathcal{T}}_2$, with vertex at the origin, such that
$$\gamma-k_2\hbox{arg}(T_2)\in\left(-\frac{\pi}{2}+\delta_2,\frac{\pi}{2}-\delta_2\right),$$
for some $\delta_2>0$ and well chosen $\gamma\in \hbox{arg}(S_d)$. Observe that, in particular, there exists $\delta_3>0$ with $\cos(\gamma-k_2\hbox{arg}(T_2))>\delta_3$.

The next technical result describes accurate bounds for the solutions of (\ref{e2}) in different domains. The proof is left to Section~\ref{secanexo2}.

\begin{prop}\label{lema361}
Let $U_\gamma(\bT,m,\epsilon)$ be defined in (\ref{e242}), with $\omega(\tau,m,\epsilon)=\omega^d(\tau,m,\epsilon)$ being the function obtained in Proposition~\ref{prop5}. The following statements hold:
\begin{enumerate}
\item There exist small enough $\rho_1>0$ and large enough $\rho_2^\infty>0$ such that 
\begin{multline}\label{e365}
|U_{\gamma}(\bT,m,\epsilon)|\le \tilde{C}_1(1+|m|)^{-\mu}e^{-\beta|m|}\left(1+|T_1|^{1/2}\exp\left(-\frac{k_1}{2\log(q)}\log^2\left(\frac{\rho}{|T_1|}\right)\right)\frac{|T_2|^{k_2}}{\delta_3}\right.\\
\left.e^{\frac{k'_1}{2\log(q)}\log^2(\frac{|T_2|^{k_2}}{\delta_3})+\alpha\log(\frac{|T_2|^{k_2}}{\delta_3})}\left[1+\left(\log(\frac{|T_2|^{k_2}}{\delta_3})\right)^{1/2}e^{\frac{k'_1}{\log(q)}\log^2\left(\frac{|T_2|^{k_2}}{\delta_3}\right)\log\left(\frac{k'_1}{\log(q)}\log\left(\frac{|T_2|^{k_2}}{\delta_3}\right)\right)}\right]\right),
\end{multline}
for every $T_1\in\tilde{\mathcal{T}}_1$ with $|T_1|<\rho_1$ and $T_2\in\tilde{\mathcal{T}}_2$ with $|T_2|>\rho^\infty_2$, and some $\tilde{C}_1>0$.
\item There exist small enough $\rho_1,\rho_2>0$ such that
\begin{multline}\label{e366}
|U_{\gamma}(\bT,m,\epsilon)|\le \tilde{C}_2(1+|m|)^{-\mu}e^{-\beta|m|}\left(1+|T_1|^{1/2}\exp\left(-\frac{k_1}{2\log(q)}\log^2\left(\frac{\rho}{|T_1|}\right)\right)\right.\\
\left.\times\exp\left(-\frac{\rho\delta_3}{2|T_2|^{k_2}}\right)\right),
\end{multline}
for every $T_1\in\tilde{\mathcal{T}}_1$ with $|T_1|<\rho_1$ and $T_2\in\tilde{\mathcal{T}}_2$ with $|T_2|<\rho_2$, and some $\tilde{C}_2>0$.
\end{enumerate}
\end{prop}

\section{Analytic solutions of the main problem: inner and outer solutions}\label{sec5}

In this section, we preserve the values of the elements involved in the main problem (\ref{e1}) stated in Section~\ref{sec2}. More precisely, we assume (\ref{e93})-(\ref{e270}), and also the hypotheses on the forcing term (\ref{e243}) in (\ref{epsi}) and the coefficients in (\ref{e284}). Let $d\in\R$ and $S_d$ an infinite sector with vertex at $0\in\C$ under the geometric condition imposed in Proposition~\ref{prop5}. Our main aim is to construct analytic solutions of (\ref{e1}) and their asymptotic behavior in different domains. For this purpose, we consider the analytic solutions as stated in Section~\ref{sec31}.

Such solutions are defined in families of sectors with respect to the perturbation parameter, conforming good coverings of $\C^\star$ (see Definition~\ref{goodcovering}). We also provide information about the difference of two solutions in consecutive sectors of the good covering, which will be crucial to determine the asymptotic behavior of the analytic solutions. We refer to consecutive solutions to solutions which are associated to consecutive elements in a fixed good covering of $\C^\star$. Let us first recall the notion of good covering in $\C^\star$.

\begin{defin}\label{goodcovering}
Let $\iota\ge2$ be an integer. For every $0\le h\le \iota-1$, we choose a finite sector with vertex at the origin $\mathcal{E}_h$ such that:
\begin{itemize}
\item $\mathcal{E}_h\subseteq D(0,\epsilon_0)$, and $\mathcal{E}_{j_1}\cap\mathcal{E}_{j_2}=\emptyset$ if and only if $0\le j_1,j_2\le \iota-1$ with  $|j_1-j_2|\ge2$ (under the convention that $\mathcal{E}_\iota:=\mathcal{E}_0$).
\item $\cup_{h=0}^{\iota-1}\mathcal{E}_h=\mathcal{U}\setminus\{0\}$, for some neighborhood of the origin $\mathcal{U}$. 
\end{itemize}
A family of sectors $(\mathcal{E}_h)_{0\le h\le \iota-1}$ under these assumptions is known as a good covering in $\C^\star$.
\end{defin}

Let $\tilde{\mathcal{T}}_1$ and $\tilde{\mathcal{T}}_2$ be sectors following the construction in Section~\ref{sec4}.

\begin{defin}
Let $(\mathcal{E}_h)_{0\le h\le \iota-1}$ be a good covering in $\C^\star$. We also fix a bounded sector, $\mathcal{T}_1$, and an unbounded sector $\mathcal{T}_2$ , both with vertex at the origin. For all $0\le h\le \iota-1$, let $S_{d_h}$ be an infinite sector of bisecting direction $d_h\in\R$. We say that the set $\{\mathcal{T}_1,\mathcal{T}_2,(\mathcal{E}_h)_{0\le h\le \iota-1},(S_{d_h})_{0\le h\le \iota-1}\}$ is admissible if the following conditions hold:
\begin{itemize}
\item For every $0\le h\le \iota-1$, $\epsilon\in\mathcal{E}_h$ and $t_1\in\mathcal{T}_1$ we have $\epsilon^{\lambda_1}t_1\in\tilde{\mathcal{T}}_1$.
\item For every $0\le h\le \iota-1$, $\epsilon\in\mathcal{E}_h$ and $t_2\in\mathcal{T}_2$ we have $\epsilon^{\lambda_2}t_2\in\tilde{\mathcal{T}}_2$.
\end{itemize}
\end{defin}

Observe that given an admissible set $\{\mathcal{T}_1,\mathcal{T}_2,(\mathcal{E}_h)_{0\le h\le \iota-1},(S_{d_h})_{0\le h\le \iota-1}\}$, the sectors $\tilde{\mathcal{T}}_1$, the choice of the sectors $\tilde{\mathcal{T}}_1$, $\tilde{\mathcal{T}}_2$ and $\gamma:=\gamma_d\in S_{d_h}$, fixed in Section~\ref{sec32a}, entail the existence of $\delta_1>0$ with 
$$\left|1+\frac{re^{\gamma_{d_h}\sqrt{-1}}}{\epsilon^{\lambda_1}t_1}\right|\ge \delta_1,\quad r\ge0, t_1\in\mathcal{T}_1,\epsilon\in\mathcal{E}_h,\quad 0\le h\le \iota-1.$$
In addition to this, it holds that
$$\gamma_{d_h}-k_2\hbox{arg}(\epsilon^{\lambda_2}t_2)\in\left(-\frac{\pi}{2}+\delta_2,\frac{\pi}{2}-\delta_2\right),\quad \epsilon\in\mathcal{E}_h,\quad 0\le h\le \iota-1,$$
for some $\delta_2>0$ and $\cos(\gamma_{d_h}-k_2\hbox{arg}(\epsilon^{\lambda_2}t_2))>\delta_3$, for some $\delta_3>0$, $\epsilon\in\mathcal{E}_h$, $0\le h\le \iota-1$.

\begin{defin}
Let $\{\mathcal{T}_1,\mathcal{T}_2,(\mathcal{E}_h)_{0\le h\le \iota-1},(S_{d_h})_{0\le h\le \iota-1}\}$ be an admissible set. 

Then, $\{(S_{d_h})_{0\le h\le \iota-1},\mathcal{T}_1\times\mathcal{T}_2\}$ is known as a family of sectors associated to the good covering $(\mathcal{E}_h)_{0\le h\le \iota-1}$.
\end{defin}

\begin{theo}\label{teo1}
Let $(\mathcal{E}_h)_{0\le h\le \iota-1}$ be a good covering in $\C^\star$. For every $0\le h\le \iota-1$ we choose $d_h\in\R$ such that $S_{d_h}$ satisfies the geometric conditions of Proposition~\ref{prop5}. Let $\{(S_{d_h})_{0\le h\le\iota-1},\mathcal{T}_1\times\mathcal{T}_2\}$ be a family of sectors associated to the good covering $(\mathcal{E}_h)_{0\le h\le \iota-1}$. If there exist small enough $\epsilon_0,C_{\psi}>0$ such that if 
$$\sup_{\epsilon\in D(0,\epsilon_0)}\left\|\psi(\tau,m,\epsilon)\right\|_{(k'_1,\beta,\mu,\alpha)}\le C_{\psi},$$
then for every $0\le h\le \iota-1$ the problem (\ref{e1}) admits a solution $u_{h}(\bt,z,\epsilon)$, which defines a bounded and holomorphic function in $\mathcal{T}_1\times\mathcal{T}_2\times H_{\beta'}\times \mathcal{E}_{h}$, for any fixed $0<\beta'<\beta$. 
\end{theo}
\begin{proof}
Let $0\le h\le \iota-1$. Proposition~\ref{prop5} guarantees the existence of $\epsilon_0,C_{\psi}>0$ such that the equation (\ref{e3}) admits a unique solution $\omega^{d_{h}}(\tau,m,\epsilon)$ which belongs to $qExp^{d_h}_{(k'_1,\beta,\mu,\alpha)}$ for every $\epsilon\in D(0,\epsilon_0)$, and the map $\epsilon\mapsto\omega^{d_h}(\tau,m,\epsilon)$ is holomorphic in $D(0,\epsilon_0)$. Taking into account that $\{(S_{d_h})_{0\le h\le\iota-1},\mathcal{T}_1\times\mathcal{T}_2\}$ is associated to the good covering $(\mathcal{E}_{h})_{0\le h\le \iota-1}$ and the properties of Laplace transform stated in Section~\ref{sec32a}, one can construct  $U_{\gamma_h}(\bT,m,\epsilon)$ in the form (\ref{e242}), which is well defined, bounded and continuous  function on $\tilde{\mathcal{T}}_1\times \tilde{\mathcal{T}}_2\times \R\times D(0,\epsilon_0)$, where $\tilde{\mathcal{T}}_1\subseteq \mathcal{R}_{\gamma_h,\tilde{\delta}}\cap D(0,r_1)$ 
for $0<r_1\le q^{(1/2-\alpha)/k_1}/2$, and $\tilde{\mathcal{T}}_2$ is an infinite sector. The function $U_{\gamma_h}(\bT,m,\epsilon)$ is holomorphic w.r.t. $(\bT,\epsilon)$ on $\tilde{\mathcal{T}}_1\times \tilde{\mathcal{T}}_2\times D(0,\epsilon_0)$. As a matter of fact, $U_{\gamma_h}$ is a solution of (\ref{e2}) in view of the properties relating both equalities in Section~\ref{secanexo}. We finally define $u_{h}(\bt,z,\epsilon)$ following (\ref{e225}):
$$u_{h}(\bt,z,\epsilon)=\mathcal{F}^{-1}(m\mapsto U_{\gamma_h}(\epsilon^{\lambda_1}t_1,\epsilon^{\lambda_2}t_2,m,\epsilon))(z),$$
which turns out to be a holomorphic solution of (\ref{e1}), defined on 
 $\mathcal{T}_1\times\mathcal{T}_2\times H_{\beta'}\times\mathcal{E}_{h}$, for any fixed $0<\beta'<\beta$.
\end{proof}

In order to provide the asymptotic behavior of the analytic solutions of (\ref{e1}) in different domains, with respect to the perturbation parameter $\epsilon$, we state the definition of inner and outer solutions of the problem (\ref{e1}).

\subsection{Inner solutions of the main problem}

\begin{defin}\label{defi7a}
Let $\iota_1\ge 2$. Let $(\mathcal{E}^\infty_{h_1})_{0\le h_1\le\iota_1-1}$ be a good covering of $\C^\star$. We also consider the admissible set $\{\mathcal{T}_1,\mathcal{T}_2,(\mathcal{E}^\infty_{h_1})_{0\le h_1\le \iota_1-1},(S^\infty_{d_{h_1}})_{0\le h_1\le \iota_1-1}\}$. Let $\mu_2>0$ be a natural number satisfying
\begin{equation}\label{e435}
\mu_2>\lambda_2,\quad k_1\lambda_1^2>k'_1((\mu_2-\lambda_2)k_2)^2.
\end{equation}

Let $\chi_2^\infty$ be a bounded domain, such that the good covering $(\mathcal{E}^\infty_{h_1})_{0\le h_1\le \iota_1-1}$ satisfies the following condition: for all $0\le h_1\le \iota_1-1$ we can select $\theta_{h_1}\in\R$ (which depends on $\mathcal{E}_{h_1}^{\infty}$) such that for every $x_2\in\chi_2^\infty$ and $\epsilon\in\mathcal{E}_{h_1}^{\infty}$, the complex number $t_2=\frac{x_2}{\epsilon^{\mu_2}}e^{\theta_{h_1}\sqrt{-1}}$ belongs to $\mathcal{T}_2$. We define the set $\mathcal{T}_{2,\epsilon,\mu_2}:=\{\frac{x_2}{\epsilon^{\mu_2}}e^{\theta_{h_1}\sqrt{-1}}:x_2\in\chi_2^\infty\}$.

In case that $\epsilon\in\mathcal{E}^{\infty}_{h_1}$, $t_2\in\mathcal{T}_{2,\epsilon,\mu_2}$, $t_1\in\mathcal{T}_1$, $z\in H_{\beta'}$ for $0<\beta'<\beta$, then we say that $u_{h_1}(\bt,z,\epsilon)$ represents an \textit{inner solution} of (\ref{e1}).
\end{defin}

\begin{theo}\label{teo2a}
Under the assumptions of Theorem~\ref{teo1} and the constraints on the inner solutions of the main problem of Definition~\ref{defi7a}, let $(\mathcal{E}^\infty_{h_1})_{0\le h_1\le\iota_1}$ be a good covering of $\C^\star$. Then, there exists $C_{inn}>0$ such that for every $0\le h_1\le \iota_1-1$ , $\epsilon\in\mathcal{E}^\infty_{h_1}\cap\mathcal{E}^\infty_{h_1+1}$, $t_{2}\in\mathcal{T}_{2,\epsilon,\mu_2}$, $z\in H_{\beta'}$ for any fixed $0<\beta'<\beta$ and $t_1\in\mathcal{T}_1$, one has
\begin{multline}\label{e449}
|u_{h_1+1}(\bt,z,\epsilon)-u_{h_1}(\bt,z,\epsilon)|\le C_{inn}|\epsilon|^{\Delta_1}\exp\left(-\frac{k''_1}{2\log(q)}\log^2\left(\frac{\hat{C}_1}{|\epsilon|^{\lambda_1}}\right)\right.\\
\left.+\frac{k'_1}{\log(q)}\log\left(\frac{\hat{C}_2}{|\epsilon|^{(\mu_2-\lambda_2)k_2}}\right)\log\left(\log(\frac{\hat{C}_3}{|\epsilon|^{(\mu_2-\lambda_2)k_2}})\right)\right),
\end{multline}
for some $\hat{C}_1,\hat{C}_2,\hat{C}_3>0$, $\Delta_1\in\R$ and a real number $k''_1\in(0,k_1)$ chosen small enough.
\end{theo}
\begin{proof}
Let $0\le h_1\le \iota_1-1$ and consider consecutive solutions $u_{h_1},u_{h_1+1}$ of (\ref{e1}), constructed in Theorem~\ref{teo1}. We recall that the function $\omega^{d_{h_1}}(\tau,m,\epsilon)$ stands for the analytic continuation of a function $\omega(\tau,m,\epsilon)$, holomorphic w.r.t. $\tau$ in a neighborhood of the origin $D(0,\rho)$, to the infinite sector $S_{d_{h_1}}$. This entails that the difference $u_{h_1+1}(\bt,z,\epsilon)-u_{h_1}(\bt,z,\epsilon)$ can be written in the following form, after an appropriate path deformation in $\tau$, avoiding the roots of $P_m(\tau)$ (see Figure~\ref{fig1}).

\begin{figure}[ht]
\begin{center}
\includegraphics[width=.45\textwidth]{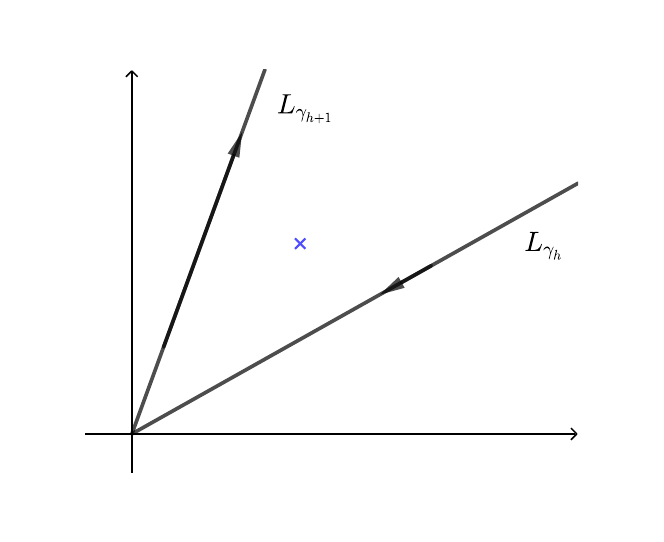}
\includegraphics[width=.45\textwidth]{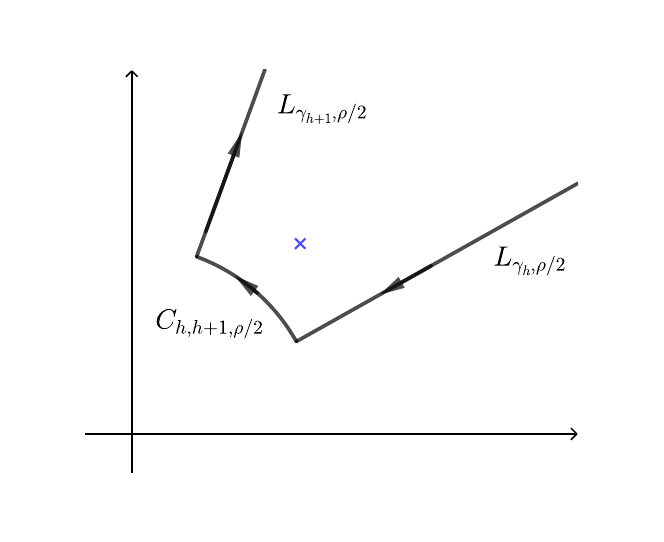}
\end{center}
\caption{Initial (left) and deformed path (right) for the difference of two consecutive solutions of (\ref{e1}). The symbol ``$\times$'' represents a root of $P_m(\tau)$}\label{fig1}
\end{figure}

We write $u_{h_1+1}(\bt,z,\epsilon)-u_{h_1}(\bt,z,\epsilon)=J_1-J_2+J_3$, where
\begin{multline}
J_1=\frac{1}{(2\pi)^{1/2}}\frac{1}{\pi_{q^{1/k_1}}}\int_{-\infty}^{\infty}\int_{L_{\gamma_{h_1+1},\rho/2}}\omega^{d_{h_1+1}}(u,m,\epsilon)\frac{1}{\Theta_{q^{1/k_1}}\left(\frac{u}{\epsilon^{\lambda_1}t_1}\right)}e^{-\left(\frac{1}{\epsilon^{\lambda_2}t_2}\right)u}\frac{du}{u}e^{izm} dm,\\
J_2=\frac{1}{(2\pi)^{1/2}}\frac{1}{\pi_{q^{1/k_1}}}\int_{-\infty}^{\infty}\int_{L_{\gamma_{h_1},\rho/2}}\omega^{d_{h_1}}(u,m,\epsilon)\frac{1}{\Theta_{q^{1/k_1}}\left(\frac{u}{\epsilon^{\lambda_1}t_1}\right)}e^{-\left(\frac{1}{\epsilon^{\lambda_2}t_2}\right)u}\frac{du}{u}e^{izm} dm,\\
J_3=\frac{1}{(2\pi)^{1/2}}\frac{1}{\pi_{q^{1/k_1}}}\int_{-\infty}^{\infty}\int_{C_{h_1,h_1+1,\rho/2}}\omega(u,m,\epsilon)\frac{1}{\Theta_{q^{1/k_1}}\left(\frac{u}{\epsilon^{\lambda_1}t_1}\right)}e^{-\left(\frac{1}{\epsilon^{\lambda_2}t_2}\right)u}\frac{du}{u}e^{izm} dm,
\end{multline}
where $L_{\gamma_{j},\rho/2}=[\rho/2e^{\gamma_{j}\sqrt{-1}},\infty)$ for $j=h_1,h_1+1$, and $C_{h_1,h_1+1,\rho/2}$ stands for the arc of circle from $\rho/2e^{\gamma_{h_1}\sqrt{-1}}$ to $\rho/2e^{\gamma_{h_1+1}\sqrt{-1}}$.

Assume that $\epsilon\in\mathcal{E}^\infty_{h_1}\cap\mathcal{E}^\infty_{h_1+1}$ and $t_2\in\mathcal{T}_{2,\epsilon,\mu_2}$, $z\in H_{\beta'}$ for some $0<\beta'<\beta$ and $t_1\in\mathcal{T}_1$. Owing to the estimates leading to (\ref{e365}), displayed in (\ref{e419}), we derive that $\hbox{dist}(\chi_{2}^{\infty}\epsilon^{\lambda_2-\mu_2},0)\ge\rho_2^\infty$ for some large enough $\rho_2^\infty$. Moreover, if $\rho_0>0$ is small enough, then there exists $\tilde{C}_8>0$ such that
\begin{multline}
|J_1|\le \tilde{C}_8|\epsilon^{\lambda_1}t_1|^{1/2}e^{-\frac{k_1}{2\log(q)}\log^2\left(\frac{\rho/2}{|\epsilon^{\lambda_1}t_1|}\right)}\frac{|\epsilon^{\lambda_2-\mu_2}x_2|^{k_2}}{\delta_3}\\
\times e^{\frac{k'_1}{2\log(q)}\log^2(\frac{|\epsilon^{\lambda_2-\mu_2}x_2|^{k_2}}{\delta_3})+\alpha\log(\frac{|\epsilon^{\lambda_2-\mu_2}x_2|^{k_2}}{\delta_3})}\left[1+\left(\log(\frac{|\epsilon^{\lambda_2-\mu_2}x_2|^{k_2}}{\delta_3})\right)^{1/2}\right.\\
\times\left.e^{\frac{k'_1}{\log(q)}\log\left(\frac{|\epsilon^{\lambda_2-\mu_2}x_2|^{k_2}}{\delta_3}\right)\log\left(\frac{k'_1}{\log(q)}\log\left(\frac{|\epsilon^{\lambda_2-\mu_2}x_2|^{k_2}}{\delta_3}\right)\right)}\right]\int_{-\infty}^{\infty}(1+|m|)^{-\mu}e^{-|m|(\beta-\beta')}d|m|\label{e488}
\end{multline}
for every $t_1\in \mathcal{T}_1$, $z\in H_{\beta'}$, $\epsilon\in\mathcal{E}^\infty_{h_1}\cap\mathcal{E}^\infty_{h_1+1}$, and $x_2\in\chi_2^{\infty}$.
An analogous upper bound is attained for $|J_2|$. Concerning $|J_3|$, one can apply the bounds stated in Proposition~\ref{prop5}, and (\ref{e190}) to get the existence of $\tilde{C}_9>0$ such that
\begin{align}
|J_3|&\le \tilde{C}_9\left(\frac{2}{\rho}\right)^{3/2}|\epsilon^{\lambda_1}t_1|^{1/2}\exp\left(-\frac{k_1}{2\log(q)}\log^2\left(\frac{\rho/2}{|\epsilon^{\lambda_1}t_1|}\right)\right)\exp\left(-\frac{\rho/2}{|\epsilon^{\lambda_2-\mu_2}x_2|^{k_2}}\delta_3\right)\nonumber\\
&\le \tilde{C}_9\left(\frac{2}{\rho}\right)^{3/2}|\epsilon^{\lambda_1}t_1|^{1/2}\exp\left(-\frac{k_1}{2\log(q)}\log^2\left(\frac{\rho/2}{|\epsilon^{\lambda_1}t_1|}\right)\right).\label{e489}
\end{align} 
The result follows taking into account (\ref{e488}) and (\ref{e489}), under the condition (\ref{e435}).
\end{proof}

\subsection{Outer solutions of the main problem}

\begin{defin}\label{defi7b}
Let $\iota_2\ge2$ and  $(\mathcal{E}_{h_2}^0)_{0\le h_2\le\iota_2}$ be a good covering of $\C^\star$, and consider an admissible set $\{\mathcal{T}_1,\mathcal{T}_2,(\mathcal{E}^0_{h_2})_{0\le h_2\le \iota_2-1},(S^0_{d_{h_2}})_{0\le h_2\le \iota_2-1}\}$. Assume that $t_2\in\mathcal{T}_2$ is such that $|t_2|<\rho_2$ for some fixed $\rho_2>0$ which is independent of $\epsilon$, then we say that $u_{h_2}(\bt,z,\epsilon)$ represents an \textit{outer solution} of (\ref{e1}).
\end{defin}

\begin{theo}\label{teo2b}
Under the assumptions of Theorem~\ref{teo1} and the constraints on the outer solutions of the main problem of Definition~\ref{defi7b}, let $(\mathcal{E}^0_{h_2})_{0\le h_2\le\iota_2}$ be a good covering of $\C^\star$. Then, there exists $C_{out}>0$ such that for every $0\le h_2\le \iota_2-1$ , $\epsilon\in\mathcal{E}^0_{h_2}\cap\mathcal{E}^0_{h_2+1}$, $t_{2}\in\mathcal{T}_{2}$ with $|t_2|<\rho_2$, $z\in H_{\beta'}$ for any fixed $0<\beta'<\beta$ and $t_1\in\mathcal{T}_1$, one has
\begin{equation}\label{e449b}
|u_{h_2+1}(\bt,z,\epsilon)-u_{h_2}(\bt,z,\epsilon)|\le C_{out}|\epsilon|^{\Delta_2}\exp\left(-\frac{k_1}{2\log(q)}\log^2\left(\frac{\hat{C}_4}{|\epsilon|^{\lambda_1}}\right)-\frac{\hat{C}_5}{|\epsilon|^{\lambda_2k_2}}\right),
\end{equation}
for some $\hat{C}_4,\hat{C}_5>0$ and $\Delta_2\in\R$.
\end{theo}
\begin{proof}
Let $0\le h_2\le \iota_2-1$ and consider consecutive solutions $u_{h_2},u_{h_2+1}$ of (\ref{e1}), constructed in Theorem~\ref{teo1}. We proceed to write the difference of two consecutive solutions in the form 
$$u_{h_2+1}(\bt,z,\epsilon)-u_{h_2}(\bt,z,\epsilon)=J_1-J_2+J_3,$$
for all $0\le h_2\le\iota_2-1$, with $h_1$ substituted by $h_2$ in the expressions of $J_1,J_2,J_3$ in the proof of Theorem~\ref{teo2a}. Let $\epsilon\in\mathcal{E}^0_{h_2}\cap\mathcal{E}^0_{h_2+1}$ and $t_2\in\mathcal{T}_{2}$, with $|t_2|<\rho_2$, $z\in H_{\beta'}$ for some $0<\beta'<\beta$ and $t_1\in\mathcal{T}_1$. Owing to analogous bounds as those leading to (\ref{e390b}) and (\ref{e391b}), we arrive at
$$|J_1|\le \tilde{C}_{10}|\epsilon^{\lambda_1}t_1|^{1/2}\exp\left(-\frac{k_1}{2\log(q)}\log^2\left(\frac{\rho}{|\epsilon^{\lambda_1}t_1|}\right)\right)\exp\left(-\frac{\rho/2}{2|\epsilon^{\lambda_2}t_2|^{k_2}}\delta_3\right),$$
for some $\tilde{C}_{10}>0$. Similar estimates hold for $|J_2|$. On the other hand, direct computations yield
$$|J_3|\le \tilde{C}_{11}\exp\left(-\frac{k_1}{2\log(q)}\log^2\left(\frac{\rho/2}{|\epsilon^{\lambda_1}t_1|}\right)\right)\left(\frac{|\epsilon^{\lambda_1}t_1|}{\rho/2}\right)^{1/2}\exp\left(-\frac{\rho/2}{|\epsilon^{\lambda_2}t_2|^{k_2}\delta_1}\right)\frac{2}{\rho},$$
for some $\tilde{C}_{11}>0$. This concludes the proof.
\end{proof}

\section{Asymptotic expansions of mixed order}\label{sec55}

This section is divided in two parts. The first part recalls some facts about $q-$asymptotic expansions and also describes $q-$asymptotic expansions which show a sub-Gevrey growth in their estimates, and related results.

In the second part of this section, we provide the existence of a formal solution of the main problem under study, written as a formal power series in the perturbation parameter, and explain the asymptotic relationship between this and the analytic solutions.

\subsection{Review on $q$-asymptotic expansions}

In the whole subsection, $(\mathbb{F},\left\|\cdot\right\|)$ stands for a complex Banach space.

We first recall the notion of $q-$Gevrey asymptotic expansions, which can be also found in~\cite{lamaq} in more detail.

\begin{defin}
Let $V$ be a bounded open sector with vertex at 0 in $\C$, $q\in\R$ with $q>1$. We also fix a positive integer $k$. We say that a holomorphic function $f:V\to\mathbb{F}$ admits the formal power series $\hat{f}(\epsilon)=\sum_{n\ge0}f_n\epsilon^n\in\mathbb{F}[[\epsilon]]$ as its $q-$Gevrey asymptotic expansion of order $1/k$ if for every open subsector $U$ of $V$, i.e. $\overline{U}\setminus\{0\}\subseteq V$, there exist $A, C>0$ such that
$$\left\|f(\epsilon)-\sum_{n=0}^{N}f_n\epsilon^n\right\|_{\mathbb{F}}\le C A^{N+1}q^{\frac{N(N+1)}{2k}}|\epsilon|^{N+1},$$
for every $\epsilon\in U$ and $N\ge 0$.
\end{defin}

Such pure $q-$asymptotic expansions have been recently studied when dealing with the asymptotic behavior of the solutions of $q-$difference-differential equations in the complex domain~\cite{dlm,lamaq2,lamaq,maq}.

We recall the definition of asymptotic expansion of mixed order as introduced in the work~\cite{maq19}.

\begin{defin}
Let $V$ be a bounded open sector of the complex plane with vertex at the origin. Let $q>1$ and $s\ge0$ be real numbers, and let $k$ be a positive integer. The function $f:V\to\mathbb{F}$ is said to admit the formal series $\hat{F}(\epsilon)=\sum_{n\ge0}f_n\epsilon^n\in\mathbb{F}[[\epsilon]]$ as its $((q,k);s)-$Gevrey asymptotic expansion if for every open subsector $U$ of $V$ (i.e. $(\overline{U}\setminus\{0\})\subseteq V$) there exist $A,C>0$ such that 
$$\left\|f(\epsilon)-\sum_{n=0}^{N}f_n\epsilon^n\right\|_{\mathbb{F}}\le CA^NN!^sq^{\frac{N(N+1)}{2k}}|\epsilon|^{N+1},$$
for all $\epsilon\in U$ and $N\ge0$.
\end{defin}

Observe that the set of functions admitting $q-$Gevrey asymptotic expansions of order $1/k$ coincide with the functions admitting $((q,k);0)-$Gevrey asymptotic expansion.

We now proceed to state a $((q,k);s)-$version of the Ramis-Sibuya Theorem. The classical statement of this result involves Gevrey asymptotic expansions, and guarantee $s$-summability of some power seires. This cohomological criterion can be found in Proposition 2~\cite{ba} and Lemma XI-2-6~\cite{hsiehsibuya}. 

This version has already been stated in the work~\cite{maq19} where a complete proof can be found therein.


\begin{theo}[$((q,k);s)-$Ramis-Sibuya Theorem]\label{teo553}

Let $(\mathcal{E}_h)_{0\le h\le \iota-1}$ be a good covering in $\C^{\star}$. For every $0\le h\le \iota-1$ we consider a holomorphic function $G_h:\mathcal{E}_h\to\mathbb{F}$ and define $\Delta_h:=G_{h+1}-G_h$, which turns out to be a holomorphic function in $Z_h:=\mathcal{E}_{h}\cap\mathcal{E}_{h+1}$ (here, $\mathcal{E}_{\iota}$ and $G_{\iota}$ stand for $\mathcal{E}_0$ and $G_{0}$, respectively). Assume that the following statements hold:
\begin{itemize}
\item $G_h$ is a bounded function in a vicinity of 0, for every $0\le h\le \iota-1$.
\item The function $\Delta_h(\epsilon)$ admits null $((q,k);s)-$Gevrey asymptotic expansion in $Z_h$ for every $0\le h\le \iota-1$, i.e. there exist $C_1,C_2>0$ such that
\begin{equation}\label{e553}
\left\|\Delta_h(\epsilon)\right\|_{\mathbb{F}}\le C_1C_2^{N}N!^sq^{\frac{N(N+1)}{2k}}|\epsilon|^{N+1},
\end{equation}
for every $\epsilon\in Z_h$ and all $N\ge0$.
\end{itemize}
Then, there exists $\hat{G}\in\mathbb{F}[[\epsilon]]$ which is the common $((q,k);s)-$Gevrey asymptotic expansion of $G_h(\epsilon)$ on $\mathcal{E}_h$, for all $0\le h\le \iota-1$.
\end{theo}

In~\cite{maq}, a $q-$Gevrey version of Ramis-Sibuya is obtained. That version is related to $q-$Gevrey asymptotic expansions of some positive order $k$, which coincide with $(q,k);0)-$Gevrey asymptotic expansions in our framework. Ramis-Sibuya theorem in that framework reads as follows:

\begin{theo}[(q-RS)]\label{teoqrs} Let $(\mathbb{F},\left\|\cdot\right\|_{\mathbb{F}})$ be a Banach space and $(\mathcal{E}_{h})_{0\le h\le\iota-1}$ be a good covering in $\C^\star$. For every $0\le h\le \iota-1$, let $G_h(\epsilon)$ be a holomorphic function from $\mathcal{E}_h$ into $\mathbb{F}$ and let the cocycle $\Delta_h(\epsilon)=G_{h+1}(\epsilon)-G_h(\epsilon)$ be a holomorphic function from $Z_h=\mathcal{E}_{h+1}-\mathcal{E}_h$ into $\mathbb{F}$ (with the convention that $\mathcal{E}_{\iota}=\mathcal{E}_0$ and $G_{\iota}=G_0$). We make further assumptions:
\begin{enumerate}
\item The functions $G_h(\epsilon)$ are bounded as $\epsilon$ tends to 0 on $\mathcal{E}_h$, for all $0\le h\le \iota-1$.
\item The function $\Delta_h(\epsilon)$ is $q-$exponentially flat of order $k$ on $Z_h$ for all $0\le h\le \iota-1$, meaning that there exist two constants $C_h^1\in\R$ and $C^2_h>0$ with
$$\left\|\Delta_h(\epsilon)\right\|_{\mathbb{F}}\le C^2_h|\epsilon|^{C_h^1}\exp\left(-\frac{k}{2\log(q)}\log^2|\epsilon|\right),$$
for all $\epsilon\in Z_h$, and all $0\le h\le\iota-1$.
\end{enumerate}
Then, there exists a formal power series $\hat{G}(\epsilon)\in\mathbb{F}[[\epsilon]]$ which is the common $q-$Gevrey asymptotic expansion of order $1/k$ of the function $G_h(\epsilon)$ on $\mathcal{E}_h$, for all $0\le h\le \iota-1$.
\end{theo}

\subsection{Asymptotic expansions for the analytic solutions of the main problem}

In this section, we preserve all the assumptions made on the elements involved in the main problem (\ref{e1}), detailed in Section~\ref{sec2}. Moreover, we depart from the geometric construction of the elements used to construct the analytic solutions of (\ref{e1}), described in Sections~\ref{sec4} and~\ref{sec5}, and under Assumption (\ref{e435}).

The next result shows that the difference of two consecutive solutions of the main problem allow the application of the $((q,k);s)-$version of Ramis-Sibuya Theorem, obtained in Theorem~\ref{teo553}, in adequate domains.

\begin{lemma}\label{lema713}
Let $k'_1,k''_1,k_2,\lambda_1,\lambda_2,\mu_2$ be positive constants with $\mu_2>\lambda_2$. For every $n\in\N$, we consider the function
$$\Psi(x)=x^n\exp\left(-\frac{k''_1\lambda_1^2}{2\log(q)}\log^2(x)+\frac{k'_1(\mu_2-\lambda_2)k_2}{\log(q)}\log(x)\log(\log(x))\right).$$
Then, it holds that
\begin{equation}\label{e715}
\Psi(x)\le \hat{C}\hat{A}^nn!^{S}q^{\frac{n(n+1)}{K}},\quad x>0,
\end{equation}
for $n\ge n_0$, where $n_0\ge1$ is a large enough integer depending on $k'_1,k''_1,k_2,\lambda_1,\lambda_2,\mu_2$ with 
\begin{equation}\label{e721}
S=\frac{k'_1(\mu_2-\lambda_2)k_2}{k''_1\lambda_1^2},\quad K=\frac{9k''_1\lambda_1^2}{7}.
\end{equation}
\end{lemma}
\begin{proof}
We make the change of variable $\tau=\log(x)$ and consider the auxiliary function $\overline{\Psi}(\tau):=\Psi(e^\tau)$. We search for the maximum of $\overline{\Psi}$ for $\tau\in\R$. It holds that
$$\overline{\Psi}'(\tau)=\Psi(e^\tau)\left(n-\frac{k''_1\lambda_1^2}{\log(q)}\tau+\frac{k'_1(\mu_2-\lambda_2)k_2}{\log(q)}(\log(\tau)+1)\right).$$
It is known that the solution of the equation $\log(t)+\log(b)t=\log(a)$, for some fixed $a,b>0$, is given by $t=W(a\log(b))/\log(b)$, where $W$ is the Lambert $W$ function. 
The maximum of $\overline{\Psi}(\tau)$ is then attained at
$$\tau=\tau_n= \frac{-k'_1(\mu_2-\lambda_2)k_2}{k''_1\lambda_1^2}W_{-1}\left(-\exp\left(-\frac{n\log(q)+k'_1(\mu_2-\lambda_2)k_2}{k'_1(\mu_2-\lambda_2)k_2}\right)\frac{k''_1\lambda_1^2}{k'_1(\mu_2-\lambda_2)k_2}\right),$$
where $W_{-1}$ stands for the $-1$-branch of Lambert $W$ function. Let 
\begin{equation}\label{e724}
A:=k'_1(\mu_2-\lambda_2)k_2,\quad C:=\frac{k''_1\lambda_1^2}{k'_1(\mu_2-\lambda_2)k_2}.
\end{equation}
Then $\tau_n$ can be written in the form $\tau_n=\frac{-1}{C}W_{-1}(-\exp(-\frac{n\log(q)}{A}-1)C)$.
Regarding Theorem 1~\cite{chatzigeorgiou}, we have
$$-1-\sqrt{2u}-u<W_{-1}(-e^{-u-1})<-1-\sqrt{2u}-\frac{2}{3}u,\quad u>0,$$
which can be applied to estimate $\tau_n$:
$$\frac{1+\sqrt{2u_0}+2/3u_0}{C}<-\frac{1}{C}W_{-1}(-e^{-u_0-1})<\frac{1+\sqrt{2u_0}+u_0}{C},\quad u_0=\frac{n\log(q)}{A}-\log(C),$$
leading to 
\begin{multline}
\frac{1}{C}\left(1+\sqrt{2(\frac{n\log(q)}{A}-\log(C))}+\frac{2}{3}(\frac{n\log(q)}{A}-\log(C))\right)<\tau_n\\
\tau_n<\frac{1}{C}\left(1+\sqrt{2\left(\frac{n\log(q)}{A}-\log(C)\right)}+\frac{n\log(q)}{A}-\log(C)\right).\label{e734}
\end{multline}
We conclude that $\overline{\Psi}(\tau)\le \overline{\Psi}(\tau_n)$. Taking into account (\ref{e734}) we arrive at 

\begin{multline}
\overline{\Psi}(\tau_n)=\exp\left(n\tau_n-\frac{k''_1\lambda_1^2}{2\log(q)}\tau_n^2+\frac{k'_1(\mu_2-\lambda_2)k_2}{\log(q)}\tau_n\log(\tau_n)\right)\\
< \hat{C}_1\hat{A}_1^n\exp\left(n^2\log(q)\left(\frac{1}{AC}-\frac{2k''_1\lambda_1^2}{9C^2A^2}\right)\xi+n\log(\frac{n\log(q)}{CA})\frac{k'_1(\mu_2-\lambda_2)k_2}{CA}\right),
\end{multline}

for every $0<\xi<1$, and some $\hat{C}_1,\hat{A}_1>0$. 
The definition of $A,C$ in (\ref{e724}) yields
$$\overline{\Psi}(\tau_n)\le \hat{C}_2\hat{A}_2^nn!^{\frac{k'_1(\mu_2-\lambda_2)k_2}{k''_1\lambda_1^2}}q^{n^2\frac{7}{9k''_1\lambda_1^2}}.$$
Statement (\ref{e715}) follows directly from the last inequality.
\end{proof}


\begin{theo}\label{teopral1}
Let the assumptions of Theorem~\ref{teo2a} hold. Let $\mathbb{F}_1$ be the Banach space of holomorphic and bounded functions on $\mathcal{T}_1\times\chi_2^{\infty}\times H_{\beta'}$. For all $0\le h_1\le \iota_1-1$, we consider the function 
\begin{equation}\label{e764}
\epsilon\mapsto u_{h_1}(t_1,\frac{x_2}{\epsilon^{\mu_2}}e^{\sqrt{-1}\theta_{h_1}},z,\epsilon),\quad \epsilon\in \mathcal{E}_{h_1}^{\infty}
\end{equation}
which defines a holomorphic and bounded function on $\mathcal{E}^{\infty}_{h_1}$, with values in $\mathbb{F}_1$. Then, there exist a formal series $\hat{u}^{\infty}(\epsilon)\in\mathbb{F}_1[[\epsilon]]$ such that for all $0\le h_1\le\iota_1-1$, the function (\ref{e764}) admits $\hat{u}^{\infty}(\epsilon)$ as its $((q,K);S)-$Gevrey asymptotic expansion on $\mathcal{E}^{\infty}_{h_1}$, for $K$ and $S$ stated in (\ref{e721}).
\end{theo}
\begin{proof}
Let $0\le h_1\le \iota_1-1$. Regarding (\ref{e449}) in Theorem~\ref{teo2a}, Lemma~\ref{lema713} and from usual estimates, we guarantee that for every $\epsilon\in \mathcal{E}^{\infty}_{h_1}\cap\mathcal{E}^{\infty}_{h_1+1}$, 
$$\sup_{t_1\in\mathcal{T}_1,x_2\in \chi_2^{\infty},z\in H_{\beta'}}|u_{h_1+1}(\bt,z,\epsilon)-u_{h_1}(\bt,z,\epsilon)|\le \Psi\left(\frac{1}{|\epsilon|}\right)|\epsilon|^n,\quad n\ge 0.$$
Taking into account Lemma~\ref{lema713}, we get that
$$\left\|u_{h_1+1}(\bt,z,\epsilon)-u_{h_1}(\bt,z,\epsilon)\right\|_{\mathbb{F}_1}\le \hat{C}\hat{A}^nn!^Sq^{\frac{n(n+1)}{K}}|\epsilon|^{n},\quad n\ge0,\epsilon\in\mathcal{E}^{\infty}_{h_1}\cap\mathcal{E}^{\infty}_{h_1+1}.$$
Theorem~\ref{teo553} states the existence of a formal power series $\hat{u}^{\infty}(\epsilon)\in\mathbb{F}_1[[\epsilon]]$ which is the common $((q,K);S)$-Gevrey asymptotic expansion of the function (\ref{e764}), as a function on $\mathcal{E}^{\infty}_{h_1}$ with values in $\mathbb{F}_1$, for every $0\le h_1\le \iota_1-1$.
\end{proof}

\begin{theo}\label{teopral2}
Let the assumptions of Theorem~\ref{teo2b} hold. Let $\mathbb{F}_2$ be the Banach space of holomorphic and bounded functions on $\mathcal{T}_1\times (\mathcal{T}_2\cap D(0,\rho_2))\times H_{\beta'}$. For every $0\le h_2\le \iota_2-1$, we consider the function
\begin{equation}\label{e805}
\epsilon\mapsto u_{h_2}(\bt,z,\epsilon),\quad \epsilon\in\mathcal{E}_{h_2}^0,
\end{equation}
which is an outer solution of (\ref{e1}), holomorphic and bounded on $\mathcal{E}_{h_2}^{0}$, with values in $\mathbb{F}_2$. Then, there exists a formal power series $\hat{u}^0(\epsilon)\in\mathbb{F}_2[[\epsilon]]$ such that the function (\ref{e805}) admits $\hat{u}^0(\epsilon)$ as its $q-$Gevrey asymptotic expansion of order $1/(k_1\lambda_1^2)$ on $\mathcal{E}_{h_2}^{0}$, for all $0\le h_2\le \iota_2-1$.
\end{theo}
\begin{proof}
Taking into account (\ref{e449b}) in Theorem~\ref{teo2b}, and Theorem~\ref{teoqrs}, we get the existence of a formal power series $\hat{u}^0(\epsilon)\in\mathbb{F}_2[[\epsilon]]$ which is the common $q$-Gevrey asymptotic expansion of the function (\ref{e805}), as a function on $\mathcal{E}^{0}_{h_2}$ with values in $\mathbb{F}_2$. This holds for all $0\le h_2\le \iota_2-1$.

\end{proof}

\section{Proof of Proposition~\ref{lema361}}\label{secanexo2}

In this section, we give proof of the technical Proposition~\ref{lema361}. We consider $U_\gamma(\bT,m,\epsilon)$ is constructed in the form (\ref{e242}).

For the first part of the proof, we take into account the property (\ref{e190}) on Jacobi Theta function, and Proposition~\ref{prop5}, to arrive at
\begin{align*}
|U_\gamma(\bT,m,\epsilon)|&\le \frac{1}{\pi_{q^{1/k_1}}}\int_0^{\infty}|\omega(re^{\gamma\sqrt{-1}})|\frac{1}{\left|\Theta_{q^{1/k_1}}\left(\frac{re^{\gamma\sqrt{-1}}}{T_1}\right)\right|}e^{-\frac{r}{|T_2|^{k_2}}\cos(\gamma-k_2\hbox{arg}(T_2))}\frac{dr}{r}\\
&\le \frac{\varpi}{\pi_{q^{1/k_1}}C_{q,k_1}\tilde{\delta}}(1+|m|)^{-\mu}e^{-\beta|m|}I(|T_1|,|T_2|),
\end{align*}
where 
\begin{multline}
I(|T_1|,|T_2|)=\int_0^\infty \exp\left(\frac{k'_1}{2\log(q)}\log^2(r+\delta)+\alpha\log(r+\delta)\right)\exp\left(-\frac{k_1}{2\log(q)}\log^2\left(\frac{r}{|T_1|}\right)\right)\\
\times\left(\frac{|T_1|}{r}\right)^{1/2}\exp\left(-\frac{r}{|T_2|^{k_2}}\delta_3\right)\frac{dr}{r}.
\end{multline}
We split $I(|T_1|,|T_2|)$ into the sum of $I_1(|T_1|,|T_2|)$ and $I_2(|T_1|,|T_2|)$, where the first element is associated to the integration in $(0,\rho)$ and the second is concern with the integration restricted to $(\rho,\infty)$, for some $\rho>0$. We study each part of the splitting:

We have 
$$I_1(|T_1|,|T_2|)\le C_{3}|T_1|^{1/2}\int_0^\rho\exp\left(-\frac{k_1}{2\log(q)}\log^2\left(\frac{r}{|T_1|}\right)\right)\frac{dr}{r^{3/2}},$$
for some $C_3>0$, which after the change of variable $r=|T_1|r'$ equals
$$C_3\int_0^{\rho/|T_1|}\exp\left(-\frac{k_1}{2\log(q)}\log^2(r')\right)\frac{dr'}{(r')^{3/2}},$$
which is bounded for every $|T_1|\in\R_+$.

On the other hand, we assume $\rho,\rho_1>0$ are such that $\rho\ge \rho_1$. Then, the positive function $\phi:[\rho,\infty)\to\R$ defined by $\phi(r)=\log^2(r/|T_1|)$ is monotone increasing on  $[\rho,\infty)$ for any choice of $|T_1|<\rho_1$. Therefore, 
\begin{equation}\label{e390}
I_{2}(|T_1|,|T_2|)\le \frac{1}{\rho^{3/2}}|T_1|^{1/2}\exp\left(-\frac{k_1}{2\log(q)}\log^2\left(\frac{\rho}{|T_1|}\right)\right)I_{2.1}(|T_2|^{k_2}/\delta_{3}),
\end{equation}
where
$$I_{2.1}(|T_2|^{k_2}/\delta_{3})=\int_0^{\infty} \exp\left(\frac{k'_1}{2\log(q)}\log^2(r+\delta)+\alpha\log(r+\delta)\right)\exp\left(-\frac{r\delta_3}{|T_2|^{k_2}}\right)dr.$$
Let $x=|T_2|^{k_2}/\delta_3$. We make the change of variable $r=x\tilde{r}$ in the last integral and arrive at
$$I_{2.1}(x)=\int_0^\infty \exp\left(\frac{k'_1}{2\log(q)}\log^2(x\tilde{r}+\delta)+\alpha\log(x\tilde{r}+\delta)\right)\exp\left(-\tilde{r}\right)xd\tilde{r}.$$
Taking into account that
$$\log^2(x\tilde{r}+\delta)=\log^2(x)+2\log(x)\log(\tilde{r}+\frac{\delta}{x})+\log^2(\tilde{r}+\frac{\delta}{x}),$$
$$\log(x\tilde{r}+\delta)=\log(x)+\log(\tilde{r}+\frac{\delta}{x}),$$ 
we get that $I_{2.1}(x)=I_{2.2}(x)+I_{2.3}(x)$, where the splitting is done on the integral by cutting the integration path into $(0,1-\delta/x)$ and $(1-\delta/x,\infty)$ for $I_{2.2}(x)$ and $I_{2.3}(x)$, respectively.
\begin{align}
I_{2.2}(x)&\le \exp\left(\frac{k'_1}{2\log(q)}\log^2(x)+\alpha\log(x)\right)x\int_0^{1-\delta/x}\exp\left(\frac{k'_1}{2\log(q)}\log^2(\tilde{r}+\frac{\delta}{x})\right)\exp\left(-\tilde{r}\right)d\tilde{r}\nonumber\\
&\le C_4\exp\left(\frac{k'_1}{2\log(q)}\log^2(x)+\alpha\log(x)\right)x\label{e391}
\end{align}
for some $C_4>0$ which does not depend on $x$. Concerning $I_{2.3}(x)$, we proceed analogously to arrive at 
\begin{align*}
I_{2.3}&\le e^{\frac{k'_1}{2\log(q)}\log^2(x)+\alpha\log(x)}x\int_{1-\frac{\delta}{x}}^{\infty} \left(\tilde{r}+\frac{\delta}{x}\right)^{\frac{k'_1\log(x)}{\log(q)}+\alpha}\exp\left(\frac{k'_1}{2\log(q)}\log^2(\tilde{r}+\frac{\delta}{x})\right)e^{-\tilde{r}}d\tilde{r}\\
&\le C_5 e^{\frac{k'_1}{2\log(q)}\log^2(x)+\alpha\log(x)}x\int_{1-\frac{\delta}{x}}^{\infty} \left(\tilde{r}+\frac{\delta}{x}\right)^{\frac{k'_1\log(x)}{\log(q)}}e^{-\frac{\tilde{r}}{2}}d\tilde{r},\\
&\le C_5 e^{\frac{\delta}{2x}}e^{\frac{k'_1}{2\log(q)}\log^2(x)+\alpha\log(x)}x2^{\frac{k'_1\log(x)}{\log(q)}+1}\int_{1/2}^{\infty}h^{\frac{k'_1\log(x)}{\log(q)}}e^{-h}dh\\
&\le C_5 e^{\frac{\delta}{2x}}e^{\frac{k'_1}{2\log(q)}\log^2(x)+\alpha\log(x)}x2^{\frac{k'_1\log(x)}{\log(q)}+1}\Gamma\left(\frac{k'_1\log(x)}{\log(q)}+1\right)
\end{align*}
for some $C_5>0$. In the last sequence of inequalities, we have made the change of variable $\tilde{r}+\frac{\delta}{x}=2h$. The application of Stirling formula $\Gamma(x+1)\sim_{x\to\infty}\sqrt{2\pi}x^{1/2}x^xe^{-x}$ leads us to
\begin{equation}\label{e392}
I_{2.3}\le  C_6 (\log(x))^{1/2}\left(\frac{k'_1\log(x)}{\log(q)}\right)^{\frac{k'_1\log(x)}{\log(q)}}e^{\frac{k'_1}{2\log(q)}\log^2(x)+\alpha\log(x)}x,
\end{equation}
for some $C_6>0$. From (\ref{e390}), (\ref{e391}) and (\ref{e392}) we conclude 
\begin{multline}\label{e419}
I_2(|T_1|,|T_2|)\le C_7|T_1|^{1/2}\exp\left(-\frac{k_1}{2\log(q)}\log^2\left(\frac{\rho}{|T_1|}\right)\right)\left[\frac{|T_2|^{k_2}}{\delta_3}e^{\frac{k'_1}{2\log(q)}\log^2(\frac{|T_2|^{k_2}}{\delta_3})+\alpha\log(\frac{|T_2|^{k_2}}{\delta_3})}\right.\\
\left.+\frac{|T_2|^{k_2}}{\delta_3}e^{\frac{k'_1}{2\log(q)}\log^2\left(\frac{|T_2|^{k_2}}{\delta_3}\right)+\alpha\log\left(\frac{|T_2|^{k_2}}{\delta_3}\right)}\left(\log(\frac{|T_2|^{k_2}}{\delta_3})\right)^{1/2}e^{\frac{k'_1}{\log(q)}\log\left(\frac{|T_2|^{k_2}}{\delta_3}\right)\log\left(\frac{k'_1}{\log(q)}\log\left(\frac{|T_2|^{k_2}}{\delta_3}\right)\right)}\right]
\end{multline}
for some $C_7>0$. The first statement of Proposition~\ref{lema361} holds. We give proof for the second statement. 

The first arguments in the proof of the first statement can be followed word by word up to the splitting of $I(|T_1|,|T_2|)$ into $I_1(|T_1|,|T_2|)+I_2(|T_1|,|T_2|)$. The quantity $I_1(|T_1|,|T_2|)$ is upper bounded by a constant for every $|T_1|>0$ and $|T_2|>0$. We now proceed to give upper estimates on $I_2(|T_1|,|T_2|)$. Let us choose $\rho_1,\rho>0$ such that $\phi(r)$ is monotone increasing on $[\rho,\infty)$. It holds that
\begin{equation}\label{e390b}
I_{2}(|T_1|,|T_2|)\le \frac{1}{\rho^{3/2}}|T_1|^{1/2}\exp\left(-\frac{k_1}{2\log(q)}\log^2\left(\frac{\rho}{|T_1|}\right)\right)I_{2.1}(|T_2|),
\end{equation}
where
\begin{align}
I_{2.1}(|T_2|)&=\int_\rho^{\infty} \exp\left(\frac{k'_1}{2\log(q)}\log^2(r+\delta)+\alpha\log(r+\delta)\right)\exp\left(-\frac{r\delta_3}{2|T_2|^{k_2}}\right)\exp\left(-\frac{r\delta_3}{2|T_2|^{k_2}}\right)dr.\nonumber\\
&\le C_8\exp\left(-\frac{\rho\delta_3}{2|T_2|^{k_2}}\right),\label{e391b}
\end{align}
for some $C_8>0$. The conclusion follows from this last upper bound.

\section{Connection of the solutions of (\ref{e2}) and (\ref{e3})}\label{secanexo}

Let $\epsilon\in D(0,\epsilon_0)$. We consider $\omega(\tau,m,\epsilon)\in q\hbox{Exp}^d_{(k'_1,\beta,\mu,\alpha)}$. Let $\gamma\in\R$ be an argument of $S_d$.

We depart from 
$$U_\gamma(\bT,z,\epsilon)=\frac{1}{\pi_{q^{1/k_1}}}\int_{L_\gamma}\omega(u,m,\epsilon)\frac{1}{\Theta_{q^{1/k_1}}\left(\frac{u}{T_1}\right)}e^{-\left(\frac{1}{T_2}\right)^{k_2}u}\frac{du}{u}.$$
Let $\tilde{\mathcal{T}}_1$ be a bounded sector and $\tilde{\mathcal{T}}_2$ be an unbounded sector, both with vertex at the origin, satisfying the assumptions in Section~\ref{sec5}. $U_\gamma(\bT,m,\epsilon)$ turns out to be a continuous function on $\tilde{\mathcal{T}}_1\times \tilde{\mathcal{T}}_2\times\R$, holomorphic with respect to $\bT$ on $\tilde{\mathcal{T}}_1\times \tilde{\mathcal{T}}_2$.

\begin{prop}
In the previous framework, the following identities hold for all $T_1\in\tilde{\mathcal{T}}_1,T_2\in\tilde{\mathcal{T}}_2,m\in\R$:
\begin{itemize}
\item[$(i)$] For every integer $\delta\ge 0$ one has
$$(T_2^{k_2+1}\partial_{T_2})^{\delta}U_{\gamma}(\bT,m,\epsilon)=\frac{1}{\pi_{q^{1/k_1}}}\int_{L_\gamma}(k_2u)^\delta\omega(u,m,\epsilon)\frac{1}{\Theta_{q^{1/k_1}}\left(\frac{u}{T_1}\right)}e^{-\left(\frac{1}{T_2}\right)^{k_2}u}\frac{du}{u}.$$
\item[$(ii)$] For every $1\le \ell_1\le D_1-1$ one has
\begin{multline}
T_1^{d_{\ell_1}}\sigma_{q;T_1}^{\delta_{\ell_1}}U_{\gamma}(\bT,m,\epsilon)\\
=\frac{1}{\pi_{q^{1/k_1}}}\int_{L_\gamma}\left[\frac{u^{d_{\ell_1}}}{(q^{1/k_1})^{d_{\ell_1}(d_{\ell_1}-1)/2}}\omega(q^{\delta_{\ell_1}-\frac{d_{\ell_1}}{k_1}}u,m,\epsilon)e^{-\left(\frac{1}{T_2}\right)^{k_2}q^{\delta_{\ell_1}-\frac{d_{\ell_1}}{k_1}}u}\right]\frac{1}{\Theta_{q^{1/k_1}}\left(\frac{u}{T_1}\right)}\frac{du}{u},
\end{multline}
\item[$(iii)$] It holds that
\begin{multline}
T_1^{d_{D_1}}\sigma_{q;T_1}^{\frac{d_{D_1}}{k_1}}U_{\gamma}(\bT,m,\epsilon)\\
=\frac{1}{\pi_{q^{1/k_1}}}\int_{L_\gamma}\left[\frac{u^{d_{D_1}}}{(q^{1/k_1})^{d_{D_1}(d_{D_1}-1)/2}}\omega(u,m,\epsilon)e^{-\left(\frac{1}{T_2}\right)^{k_2}u}\right]\frac{1}{\Theta_{q^{1/k_1}}\left(\frac{u}{T_1}\right)}\frac{du}{u},
\end{multline}
\item[$(iv)$] For every $1\le \ell_1\le D_1-1$ and $1\le \ell_2\le D_2-1$ we have
\begin{multline}T_1^{d_{\ell_1}}\sigma_{q;T_1}^{\delta_{\ell_1}}\sigma_{q;T_2}^{\frac{1}{k_2}\left(\delta_{\ell_1}-\frac{d_{\ell_1}}{k_1}\right)}U_{\gamma}(\bT,m,\epsilon)\\
=\frac{1}{\pi_{q^{1/k_1}}}\int_{L_\gamma}\left[\frac{u^{d_{\ell_1}}}{(q^{1/k_1})^{d_{\ell_1}(d_{\ell_1}-1)/2}}\omega(q^{\delta_{\ell_1}-\frac{d_{\ell_1}}{k_1}}u,m,\epsilon)e^{-\left(\frac{1}{T_2}\right)^{k_2}u}\right]\frac{1}{\Theta_{q^{1/k_1}}\left(\frac{u}{T_1}\right)}\frac{du}{u}.
\end{multline}
\end{itemize}
\end{prop}
\begin{proof}
Statement $(i)$ is a consequence of the derivation under the integral sign. Statement $(ii)$ is a direct consequence of Proposition 6 in~\cite{maq} with $k:=k_1$, $\sigma:=d_{\ell_1}$, $j:=\delta_{\ell_1}$ and $f(x)=\omega(x,m,\epsilon)\exp(-\left(\frac{1}{T_2}\right)^{k_2}x)$. The third statement is a consequence of Proposition 6 in~\cite{maq}. The last statement is derived from the application of $(ii)$ and the fact that 
$$\exp\left(-\left(T_2q^{\frac{1}{k_2}(\delta_{\ell_1}-\frac{d_{\ell_1}}{k_1})}\right)^{-k_2}q^{\delta_{\ell_1}-\frac{d_{\ell_1}}{k_1}}u\right)=\exp\left(-\left(T_2\right)^{-k_2}u\right).$$
\end{proof}

\end{document}